\documentclass[11pt,a4paper,oneside]{article}

\usepackage[T1]{fontenc}
\usepackage[utf8]{inputenc}
\usepackage{lmodern}
\usepackage[
  a4paper,
  margin=25mm,
  footskip=12mm
]{geometry}
\usepackage{microtype}
\usepackage{indentfirst}

\usepackage{amsmath,amssymb,amsthm}
\numberwithin{equation}{section}
\DeclareMathOperator{\Gr}{Gr}
\DeclareMathOperator{\Sp}{Sp}
\DeclareMathOperator{\GL}{GL}

\theoremstyle{plain}
\newtheorem{theorem}{Theorem}[section]
\newtheorem{lemma}[theorem]{Lemma}
\newtheorem{proposition}[theorem]{Proposition}
\newtheorem{corollary}[theorem]{Corollary}
\theoremstyle{definition}
\newtheorem{definition}[theorem]{Definition}
\theoremstyle{remark}
\newtheorem{remark}[theorem]{Remark}
\theoremstyle{plain}
\newtheorem*{fact}{Fact}

\usepackage{needspace}
\AddToHook{env/theorem/before}{\Needspace{6\baselineskip}}
\AddToHook{env/lemma/before}{\Needspace{5\baselineskip}}
\AddToHook{env/proposition/before}{\Needspace{5\baselineskip}}
\AddToHook{env/corollary/before}{\Needspace{5\baselineskip}}
\AddToHook{env/definition/before}{\Needspace{4\baselineskip}}

\usepackage{graphicx}
\graphicspath{{figures/}}
\usepackage[font=small,labelfont=bf,skip=7pt]{caption}
\usepackage{booktabs,longtable,array}
\usepackage{siunitx}
\usepackage{enumitem}
\setlist[enumerate]{label=(\arabic*),leftmargin=*,itemsep=0.25\baselineskip}

\usepackage{titling}

\pretitle{\centering\Large\bfseries}
\posttitle{\par\vskip0.6em}

\preauthor{\centering\normalsize\lineskip0.5em}
\postauthor{\par\vskip0.5em}

\predate{}
\postdate{}
\date{}

\usepackage[hidelinks,unicode]{hyperref}
\usepackage{bookmark}
\hypersetup{
  pdftitle={Linear Stability of Double-Double Orbits in the Equal-Mass Four-Body Problem via index theory},
  pdfauthor={Zijian Meng; Yuwei Ou},
  pdfkeywords={Maslov-type Index Theory, Four-Body Problem, Linear Stability, variational characterization},
  bookmarksnumbered=true,
  bookmarksopen=true,
  bookmarksopenlevel=1,
  pdfdisplaydoctitle=true,
  pdfstartview=FitH
}

\title{Linear Stability of Double-Double Orbits in the Equal-Mass Four-Body Problem via Index Theory}
\author{Zijian Meng\thanks{email:202311792@mail.sdu.edu.cn}\\[0.7em]
  {\small Department of Mathematics, Shandong University Jinan, Shandong 250100}\\[0.2em]
  {\small The People's Republic of China}}
\date{}
\begin{document}
\maketitle
\begin{abstract}
In this paper, using the Maslov-type index theory in symplectic geometry, we study the linear stability of the double-double orbits in the planar equal-mass four body problem. By analyzing its symmetry properties and variational characterization, we clarify the relationship between linear stability and its variational nature on various loop spaces and show that for any rotation angle $\theta=\frac{\pi}{n}$, the retrograde double-double orbits are linearly stable with integer $n\in[3,30]$ and the prograde double-double orbits are linearly stable with integer $n\in[4,30]$. We also give the distributions of the eigenvalues on the unit circle of the monodromy matrices associated with the orbits. The basic idea is as follows: the variational characterization of the orbits provides information about its Morse index; based on its relation to the Maslov-type index, our stability criteria come into play.
\end{abstract}

\noindent\textbf{Keywords:} Maslov-type Index Theory, Four-Body Problem, Linear Stability, variational characterization.
\begingroup
\small
\tableofcontents
\endgroup
\medskip
\section{Introdution and Main Results}\label{sec:introduction}
The planar equal-mass four-body problem studies the motion of four equal masses \((m_{1} = m_{2} = m_{3} =m_4= 1)\) under the Newtonian gravitational law. Let \(x_{i} \in \mathbb{R}^{2}\) \((i = 1,2,3,4)\) be the position of each body. The equations of motion are
\begin{equation}
\ddot{x}_{i} = \frac{\partial U}{\partial x_{i}} (x_{1},x_{2},x_{3},x_4),\quad x_{i}\in \mathbb{R}^{2},\quad i = 1,2,3,4, \label{28}
\end{equation}
where
\[
U = \sum_{1\leq i< j\leq 4}\frac{1}{|x_{i} - x_{j}|}
\]
is the potential energy and \(|\cdot |\) denotes the Euclidean norm in \(\mathbb{R}^{2}\). Define the configuration space to be \(\widehat{\mathcal{X}} = \mathcal{X} \backslash \Delta\), where
\[
\mathcal{X} \equiv \left\{x(t) = (x_{1}^{\top}(t),x_{2}^{\top}(t),x_{3}^{\top}(t),x_{4}^{\top}(t))^{\top} = \begin{pmatrix} x_{1}(t) \\ x_{2}(t) \\ x_{3}(t)\\x_{4}(t) \end{pmatrix}\in (\mathbb{R}^{2})^{4} \;\middle|\; x_{1} + x_{2} + x_{3}+x_{4} = 0\right\},
\]
and \(\Delta\) is the collision set
\[
\Delta \equiv \bigcup \{x\in \mathcal{X} \mid x_{i} = x_{j}\text{ for some }i\neq j\}.
\]
Let \(W^{1,2}(\mathbb{R} / T\mathbb{Z},\widehat{\mathcal{X}})\) be the loop space for some fixed period \(T\). The action functional of the system \eqref{28} is
\begin{align}
    \mathcal{A}:W^{1,2}(\mathbb{R} / T\mathbb{Z},\widehat{\mathcal{X}})&\to \mathbb{R},\notag\\
x&\mapsto \mathcal{A}(x) = \int_0^T L(x,\dot{x})\,\mathrm{d}t,
\label{29}
\end{align}
where
\[
L(x,\dot{x}) = \frac{1}{2} |\dot{x} (t)|^2 +U(x(t))
\]
is the Lagrangian. Set
\[
z(t) = \bigl(\dot{x}^{\top}(t),x^{\top}(t)\bigr)^{\top} \equiv \bigl(y(t)^{\top},x(t)^{\top}\bigr)^{\top}.
\]
The Hamiltonian system corresponding to the Newtonian system \eqref{28} is
\begin{align}
\dot{z} (t) &= JH'(z(t)),\label{eq:1-3}\\
z(0)&=z(T),
\label{eq:1-4}
\end{align}
where
\[
J = \begin{bmatrix} 0 & - I_n\\ I_n & 0 \end{bmatrix},
\]
$I_n$ is the \(n\times n\) identity matrix, and the Hamiltonian energy
\[
H(z) = \frac{1}{2} |y(t)|^2 - U(x(t))
\]
is defined on the phase space \(T^{*}(\widehat{\mathcal{X}})\).

The idea of minimizing the action functional in the \(N\)-body problem started by Poincaré. In 1896, Poincaré searched for periodic orbits with a given period \(T\) of the \(N\)-body problem as action minimizers. He found that it is hard to eliminate the possible existence of collisions due to the difficulty of the Newtonian potential: the action of a solution can be finite even when some of the bodies collide. Until recently in 2000, Chenciner and Montgomery \cite{CCM2000} made a breakthrough by considering minimization of action in some symmetric loop space. They successfully showed the existence of the figure-eight orbit which was numerically discovered by Moore \cite{Moore1993}. The symmetry of the loop plays an important role in avoiding collisions. Inspired by their work, a large number of symmetric periodic solutions were proved to exist by variational method, see, for example, \cite{COX2012,Chen2008,CL2009,FT2004,OX2015} and references therein.

In 2009, Hu, Long and Sun et al. \cite{HS2009,HS2010,HLS2014} successfully applied the index theory to study the linear stability of periodic orbits in \(N\)-body problem. Especially in \cite{HS2009}, Hu and Sun established a generalized Bott-type iteration formula, and together with the Morse index and Maslov-type index theory, they obtained the linear stability of the figure-eight orbit \cite{CCM2000}, which has a \(D_{6}\) symmetry. Note that the double-double orbits  admits a \(D_{n}\) symmetry \cite{Peng2015}, where $n\geq2$ is integer. The goal of this paper is to understand the linear stability of the double-double orbits from the perspective of index theory. Concretely, we shall attempt to combine the generalized Bott-type iteration formula in \cite{HS2009} with the Morse index and Maslov-type index theory.

The double-double orbits was first found numerically by Vanderbei \cite{Van04} in 2003. In his list, there is an interesting class of orbits in the parallelogram four-body problem with equal-masses, which he called double-double orbits. Actually, he found two sets of double-double orbits: retrograde double-double orbits and prograde double-double orbits. A double-double orbit is called retrograde if one pair of the two adjacent masses revolves around the mass center of the other pair in different direction. A double-double orbit is called prograde if both revolutions follow the same direction. Sample pictures of retrograde and prograde double-double orbits are shown in Figure~\ref{fig:double-double-examples}.
\begin{figure}[htbp]
\centering
\includegraphics[width=0.38\linewidth]{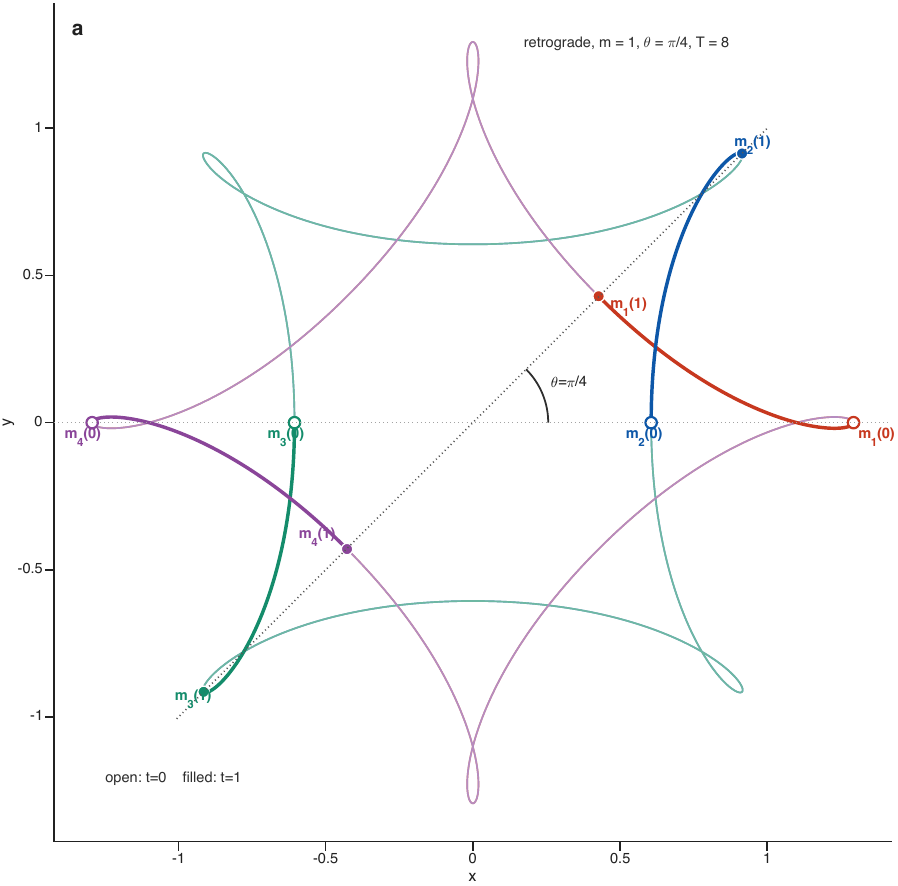}\hspace{0.06\linewidth}
\includegraphics[width=0.38\linewidth]{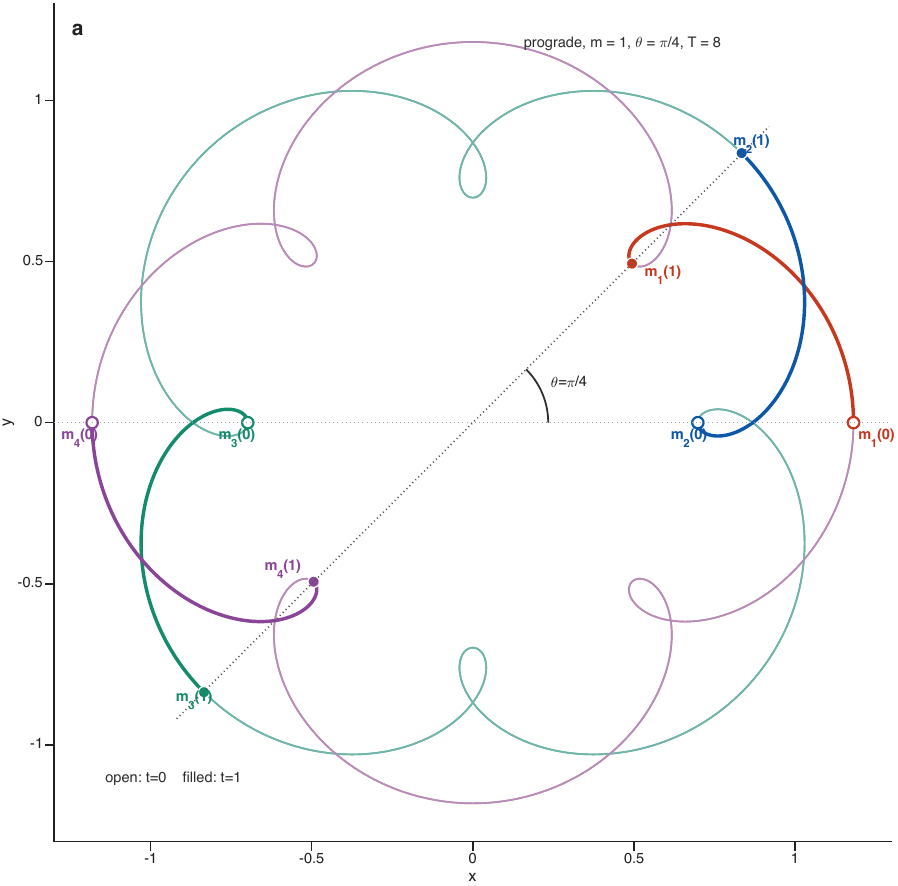}
\caption{Two sets of double-double orbits with equal masses: a retrograde orbit in left, and a prograde orbit in right. In both pictures, $\theta$ is the angle between the collinear configuration at $t=0$ and the collinear configuration at $t=1$, where the paths connecting the two configurations are highlighted. Both pictures are for $\theta=\frac{\pi}{4}$.}
\label{fig:double-double-examples}
\end{figure}

In such an orbit, two pairs of bodies locate on two sides of the origin. In each pair, the two bodies switch orders periodically. For example, as in Figure~\ref{fig:double-double-examples}, at $t=0$, the four bodies line up on the $x$-axis with an order $m_1\xrightarrow{}m_2\xrightarrow{}m_3\xrightarrow{}m_4$ from right to left. At $t=1$, they are on a line which is an anti-clockwise rotation of $\theta$ of the $x$-axis. The order of the bodies becomes $m_2\xrightarrow{}m_1\xrightarrow{}m_4\xrightarrow{}m_3$ from right to left. At $t=2$, they line up for the third time and the order becomes $m_1\xrightarrow{}m_2\xrightarrow{}m_3\xrightarrow{}m_4$ from right to left, which is the same as it is at $t=0$.

The mathematical existance of the double-double orbits has also been studied, which can be traced back to the variational studies of Chen. In the equal-mass parallelogram four-body problem, Chen \cite{Chen2001} first constructed a collision-free periodic orbit as an action minimizer on a suitable symmetric path space. He later proved the existence of a nontrivial double choreographic solution with three-fold rotational symmetry using binary decomposition and action estimates \cite{ChenBin}. In a subsequent work, Chen introduced a one-parameter family of symmetric variational problems characterized by \(d\)-fold rotational symmetry and proved the existence of classical minimizers for \(d\ge d_0\), where \(d_0\le3.55\). Consequently, all integer cases \(d\ge4\) are covered by a unified variational argument. The cases \(d=2\) and \(d=3\) were verified separately in that work by numerical action estimates, although independent variational existence results for these two special cases were already available from his earlier papers. These results provide the variational existence foundation for the retrograde double–double orbits subsequently studied by Yan et al \cite{Peng2015}.

 Later, Ferrario and Terracini \cite{FT2004} introduced a general variational argument and the existence of many interesting periodic orbits in the $N$-body problem. The equal-mass retrograde double-double orbits can be seen as one of their many applications. In 2018, Kuang and Yan \cite{KY2018}
show the existence of a set of prograde double-double orbits for any rotation angle $\theta \in(0,\frac{\pi}{4}]$ in the equal-mass four-body problem.

The linear stability of some of the double-double orbits is claimed by Vanderbei first \cite{Van04}. Actually, by testing several double-double orbits on a simulator for a relatively long time, he concluded that they are apparently stable. In 2015, Ouyang, Peng, Xu and Yan \cite{Peng2015} study the linear stability of a two-parameter family $(m,\theta)$ of double-double orbits in the parallelogram four body problem. This family was numerically found and when $\theta=\frac{\pi}{n}$ with integer $n\in[2,20]$, the retrograde double-double orbits with $m=1$ which existence has been proved are linearly stable numerically. When $\theta=\frac{\pi}{n}$ with integer $n\in[4,20]$, the prograde double-double orbits with $m=1$ which existence has been proved are numerically proved to be linearly stable. For the periodic members considered here, Chen’s rotational parameter \(d\) corresponds to the integer \(n\) in the parametrization \(\theta=\pi/n\) used by Yan et al.

Because their proof is computer assisted, in this paper, we try to understand why this linear stability is possible from a variational viewpoint by index theory of the Hamiltonian system. The key property of the double-double orbits is that it has $D_n$ full symmetry \cite{Peng2015}. Let $\widehat{\mathcal{X}}$ be the configuration space of the planar four body problem with center-of-masses fixed, for double-double orbits, we analyze their stability by distinguishing between even and odd of the integer $n$. We denote by \( m_1, m_2, m_3,m_4 \) the Morse indexs of the double-double as the critical point of the action functional in the total loop space \( W^{1,2}(\mathbb{R}/\mathbb{Z}, \widehat{\mathcal{X}}) \) with some \( \mathbb{Z}_2 \)- ,\( \mathbb{Z}_{n/2}\)- and \( \mathbb{Z}_n \)-invariant loop subspaces if $n$ is even and \( m_1,  m_3,m_4 \) the Morse indexs of the double-double as the critical point of the action functional in the total loop space \( W^{1,2}(\mathbb{R}/\mathbb{Z}, \widehat{\mathcal{X}}) \), its $[0,4]$ and \( \mathbb{Z}_n \)-invariant loop subspaces if $n$ is odd. The main results of this paper is as follows.

\begin{theorem}\label{thm:retrograde-even}
For the retrograde double double orbits, when $n\geq4$ is even, assume that
\begin{enumerate}[label=(\arabic*)]
\item $m_1=2n-4,m_2=n-2,m_3=2,m_4=0$;
\item after eliminating the center of mass and the total linear momentum, the multiplicity of the eigenvalue 1 in the monodromy matrix $\gamma(2T)$ of the retrograde double-double orbits is exactly 4;
\item the symplectic Jordan form corresponding to the angular momentum of the retrograde double-double orbits is \(\begin{bmatrix} 1 & 1\\ 0 & 1 \end{bmatrix}\).
\end{enumerate}
Then the retrograde double-double orbits with rotation angle $\theta=\frac{\pi}{n}$ where $n=2n'$ and integer $n'\in[2,15]$ are linearly stable.
\end{theorem}
For the retrograde double double orbit with $n=2$, assume that $m_1=m_3=2,m_2=m_4=0$ together with $(2)$ and $(3)$ of Theorem~\ref{thm:retrograde-even}, the case is rather special and we can refer to Remark~\ref{rem:retrograde-n2}.
\begin{theorem}\label{thm:retrograde-odd}
 For the retrograde double double orbits, when $n\geq3$ is odd, assume that
\begin{enumerate}[label=(\arabic*)]
\item $m_1=2n-4,m_3=2,m_4=0$;
\item after eliminating the center of mass and the total linear momentum, the multiplicity of the eigenvalue 1 in the monodromy matrix $\gamma(2T)$ of the retrograde double-double orbits is exactly 4;
\item the symplectic Jordan form corresponding to the angular momentum of the retrograde double-double orbits is \(\begin{bmatrix} 1 & 1\\ 0 & 1 \end{bmatrix}\).
\end{enumerate}
Then the retrograde double-double orbits with rotation angle $\theta=\frac{\pi}{n}$ where $n=2n'+1$ and integer $n'\in[1,14]$ are linearly stable.
\end{theorem}
\begin{theorem}\label{thm:prograde-even}
For the prograde double double orbits, when $n\geq6$ is even, assume that
\begin{enumerate}[label=(\arabic*)]
\item $m_1=2n,m_2=n,m_3=4,m_4=2$;
\item after eliminating the center of mass and the total linear momentum, the multiplicity of the eigenvalue 1 in the monodromy matrix $\gamma(2T)$ of the prograde double-double orbits is exactly 4;
\item the symplectic Jordan form corresponding to the angular momentum of the prograde double-double orbits is \(\begin{bmatrix} 1 & 1\\ 0 & 1 \end{bmatrix}\).
\end{enumerate}
Then the prograde double-double orbits with rotation angle $\theta=\frac{\pi}{n}$ where $n=2n'$ and integer $n'\in[3,15]$ are linearly stable.

In particular, when $n=4$, assume that
\begin{enumerate}[label=(\arabic*)]
\item $m_1=10,m_2=m_3=4,m_4=2$;
\item after eliminating the center of mass and the total linear momentum, the multiplicity of the eigenvalue 1 in the monodromy matrix $\gamma(2T)$ of the prograde double-double orbits is exactly 4;
\item the symplectic Jordan form corresponding to the angular momentum of the prograde double-double orbits is \(\begin{bmatrix} 1 & 1\\ 0 & 1 \end{bmatrix}\).
\end{enumerate}
Then the prograde double-double orbits with rotation angle $\theta=\frac{\pi}{4}$ are linearly stable.
\end{theorem}

\begin{theorem}\label{thm:prograde-odd}
    For the prograde double double orbits, when $n\geq5$ is odd, assume that
\begin{enumerate}[label=(\arabic*)]
\item $m_1=2n,m_3=4,m_4=2$;
\item after eliminating the center of mass and the total linear momentum, the multiplicity of the eigenvalue 1 in the monodromy matrix $\gamma(2T)$ of the prograde double-double orbits is exactly 4;
\item the symplectic Jordan form corresponding to the angular momentum of the prograde double-double orbits is \(\begin{bmatrix} 1 & 1\\ 0 & 1 \end{bmatrix}\).
\end{enumerate}
Then the prograde double-double orbits with rotation angle $\theta=\frac{\pi}{n}$ where $n=2n'+1$ and integer $n'\in[2,14]$ are linearly stable.
\end{theorem}
\begin{remark}
    In fact, stability analysis of the orbits can also be studied for $n\geq30$ with the same method. However, we only consider the case with $n\leq30$
    in this paper.
\end{remark}
In Theorem~\ref{thm:retrograde-even}--Theorem~\ref{thm:prograde-odd}, assumption (2) is usually referred to as the non-degeneracy assumption. Assumptions (1)--(3) are numerically checked by MATLAB. Let the monodromy matrix be \(\gamma (T)\). We numerically calculate eigenvalues of the matrix $M=\hat{S}_f^{\top}\gamma(\frac{T}{n})$, in which \(M^{n} = \gamma (T)\) and \(\hat{S}_f^{\top}\) is defined in Section~\ref{sec:symmetry}. Assumption (3) is verified by taking the symplectic
inner product of all eigenvectors and generalized eigenvectors corresponding to 1. Numerically, one can find that \(M\) and \(\gamma (T)\) has eigenvalue \(1\) with algebraic multiplicity \(4\), and the corresponding normal form is \(\begin{bmatrix} 1 & 1\\ 0 & 1 \end{bmatrix} \diamond \begin{bmatrix} 1 & 1\\ 0 & 1 \end{bmatrix}\). The main theorem will be proved by applying the generalized Bott-type iteration formula and the index theory, which can be found in Sections~\ref{sec:common-indices} and \ref{sec:stability}.

The paper is organized as follows. Section~\ref{sec:symmetry} analyzes symmetry information of the double-double orbits. Section~\ref{sec:jordan} studies the symplectic decomposition of its monodromy matrix corresponding to the Hamiltonian energy and the angular momentum. Section~\ref{sec:common-indices} gives the Morse index and Maslov-type index formulas, the generalized Bott-type iteration formula and the arguments using splitting numbers for both types of orbits. In Section~\ref{sec:stability}, these formulas are applied to the retrograde and prograde double-double orbits, including the special cases $n=2,3,4$. For readers' convenience, we briefly review the Maslov-type index theory and the generalized Bott-type iteration formula in Appendix~\ref{app:index-theory} and introduce how to compute the Morse index by using numerical methods in Appendix~\ref{app:numerical-index}.
\section{Symmetry of the double-double orbits}\label{sec:symmetry}

We consider the double-double orbits to the  Case  $M = [1, 1, 1, 1]$
and analyze their linear stabilities with respect to the  parameters:  rotation
angle $\theta$ by Maslov index theory.

Geometrically, in such an orbit, two pairs of bodies locate on two sides of the origin. In each pair, the two bodies switch orders periodically. For example, as in Figure~\ref{fig:retrograde-family} and Figure~\ref{fig:prograde-family}, at $t=0$, the four bodies line up on the $x$-axis with an order $m_1\xrightarrow{}m_2\xrightarrow{}m_3\xrightarrow{}m_4$ from right to left. At $t=1$, they are on a line which is an anti-clockwise rotation of $\theta$ of the $x$-axis. The order of the bodies becomes $m_2\xrightarrow{}m_1\xrightarrow{}m_4\xrightarrow{}m_3$ from right to left. At $t=2$, they line up for the third time and the order becomes $m_1\xrightarrow{}m_2\xrightarrow{}m_3\xrightarrow{}m_4$ from right to left, which is the same as it is at $t=0$. If $\theta= \pi/n $ with $n$ an integer, the orbit has
period $2n$. Take the function
space to be
\[E=\{z(t)\in W^{1,2}(\mathbb{R}/T \mathbb{Z}, \widehat{\mathcal{X}})\mid z(t)=z(t+T)\}.\]
The next lemma is from \cite{Peng2015}.
\begin{figure}[htbp]
\centering
\includegraphics[width=0.30\linewidth]{retrograde_m1_n4.pdf}\hfill
\includegraphics[width=0.30\linewidth]{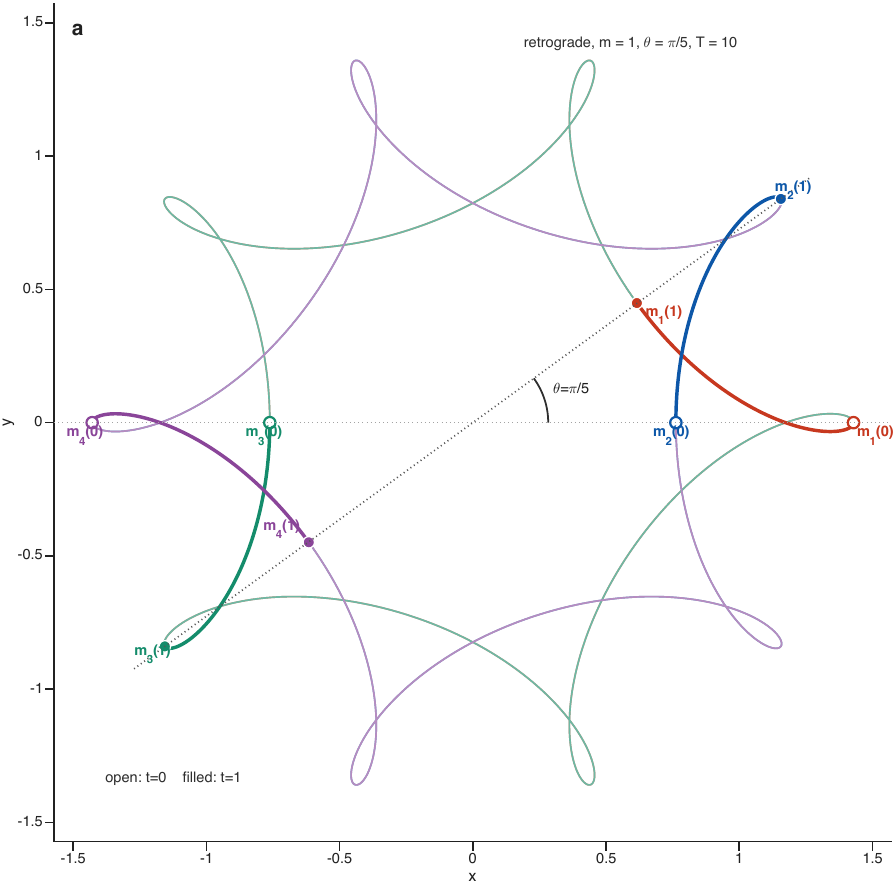}\hfill
\includegraphics[width=0.30\linewidth]{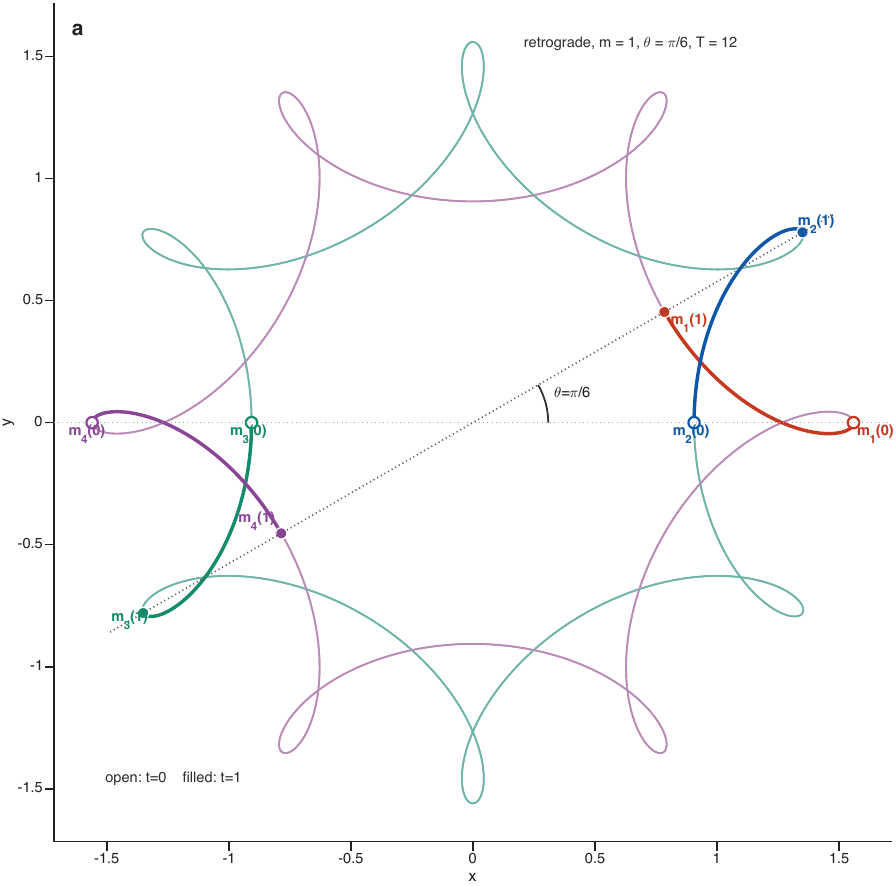}
\par\medskip
\includegraphics[width=0.30\linewidth]{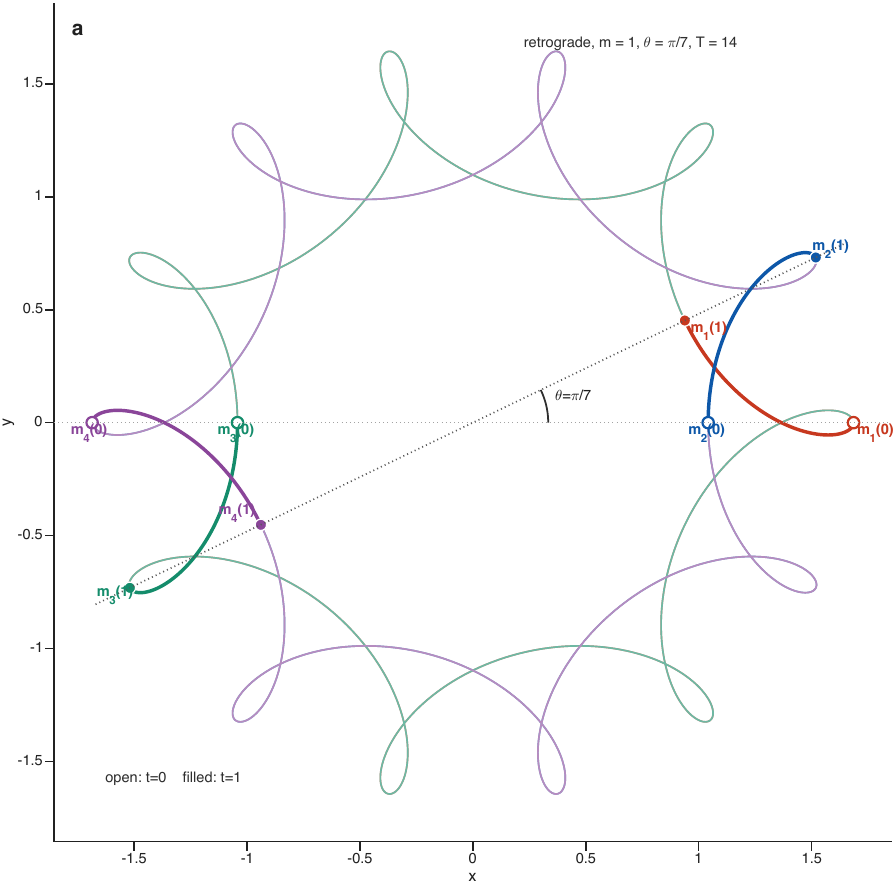}\hfill
\includegraphics[width=0.30\linewidth]{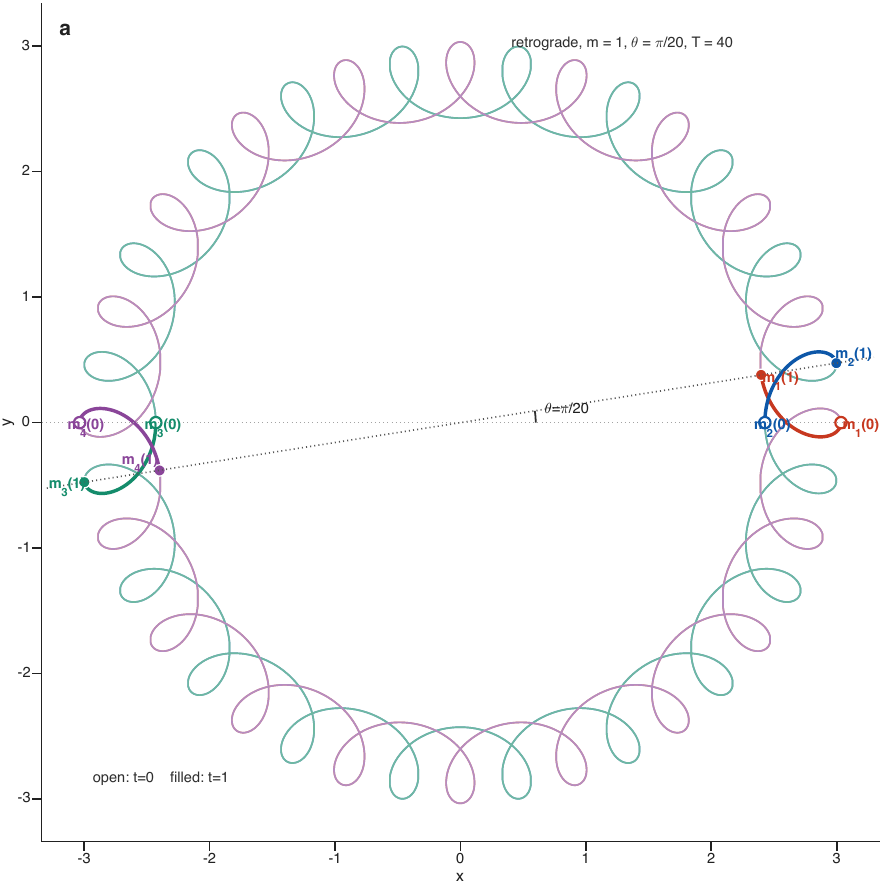}\hfill
\includegraphics[width=0.30\linewidth]{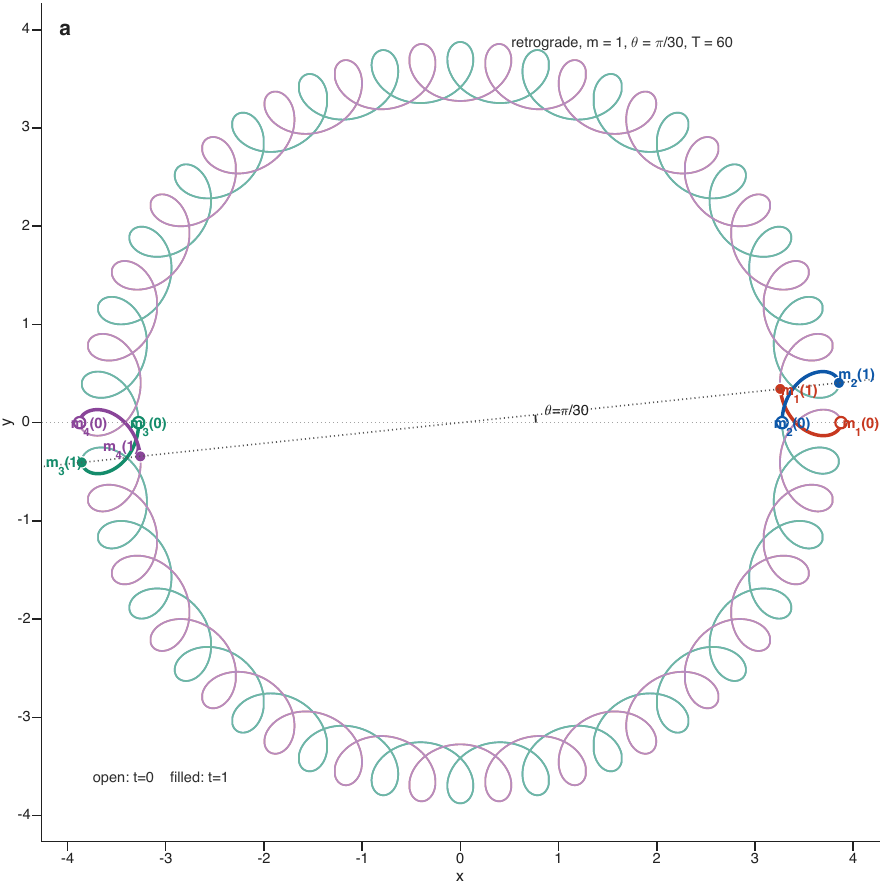}
\caption{Six sets of retrograde double-double orbits with equal masses and $\theta=\frac{\pi}{4},\frac{\pi}{5},\frac{\pi}{6},\frac{\pi}{7},\frac{\pi}{20},\frac{\pi}{30}$.}
\label{fig:retrograde-family}
\end{figure}

\begin{figure}[htbp]
\centering
\includegraphics[width=0.30\linewidth]{prograde_m1_n4.pdf}\hfill
\includegraphics[width=0.30\linewidth]{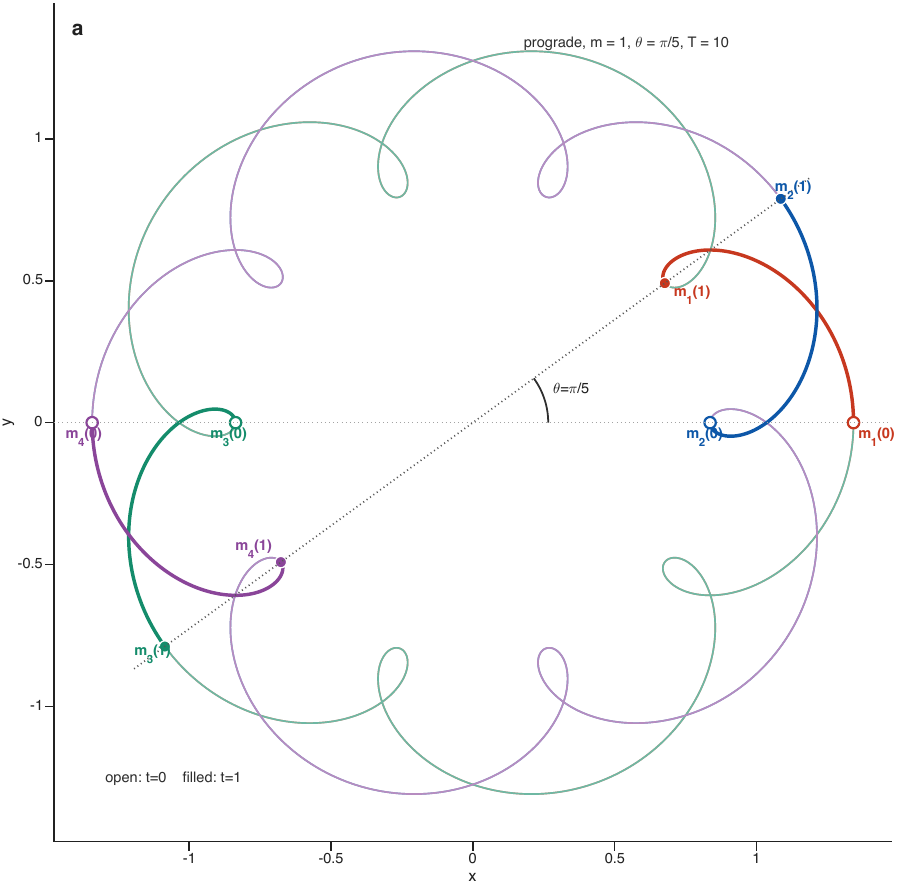}\hfill
\includegraphics[width=0.30\linewidth]{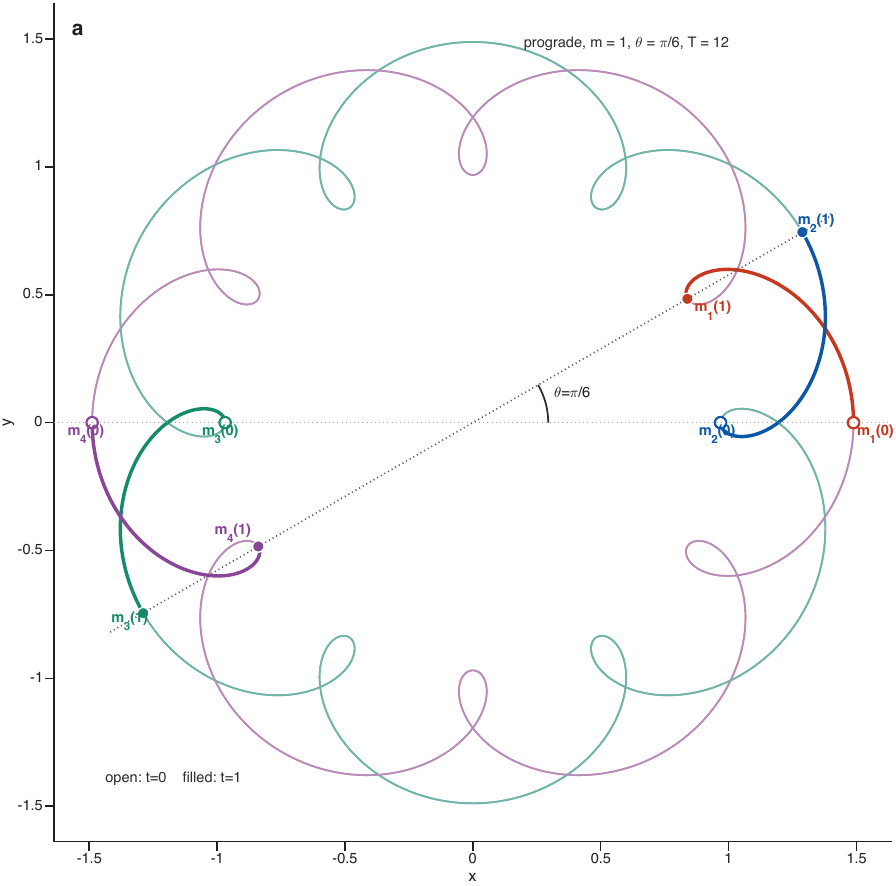}
\par\medskip
\includegraphics[width=0.30\linewidth]{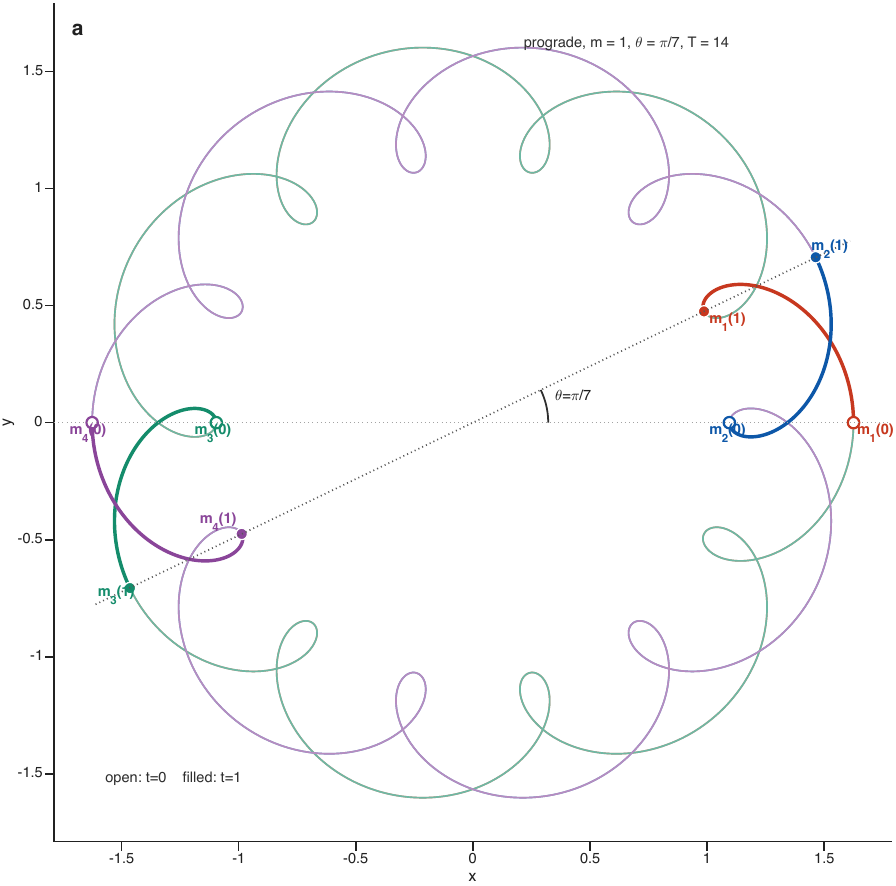}\hfill
\includegraphics[width=0.30\linewidth]{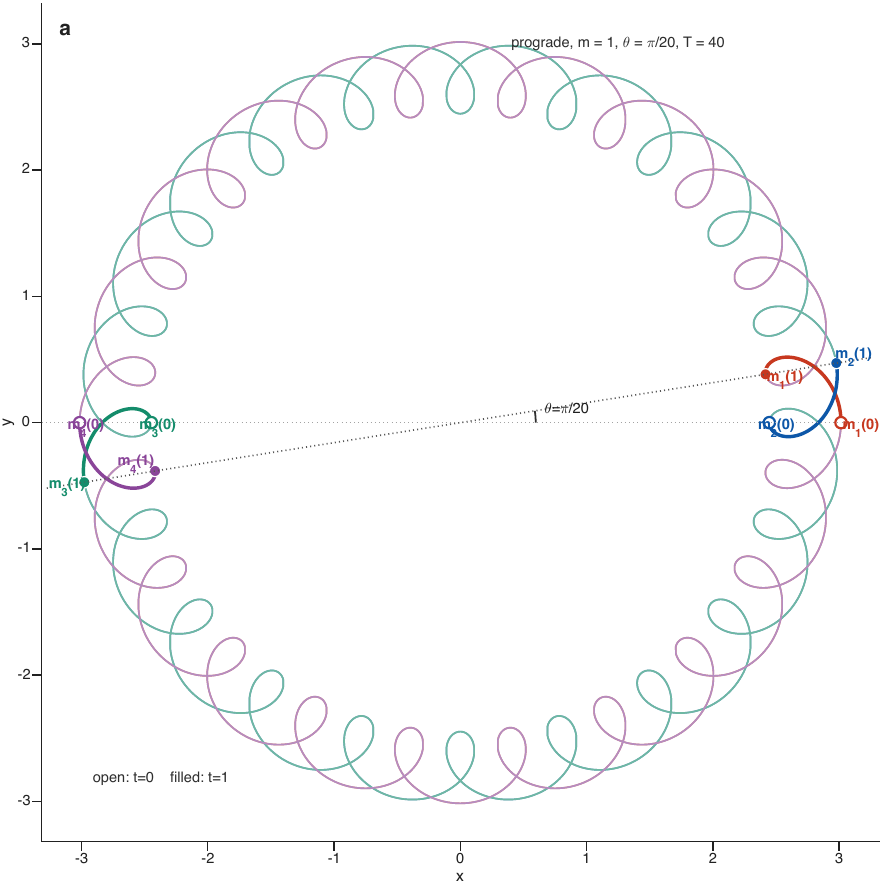}\hfill
\includegraphics[width=0.30\linewidth]{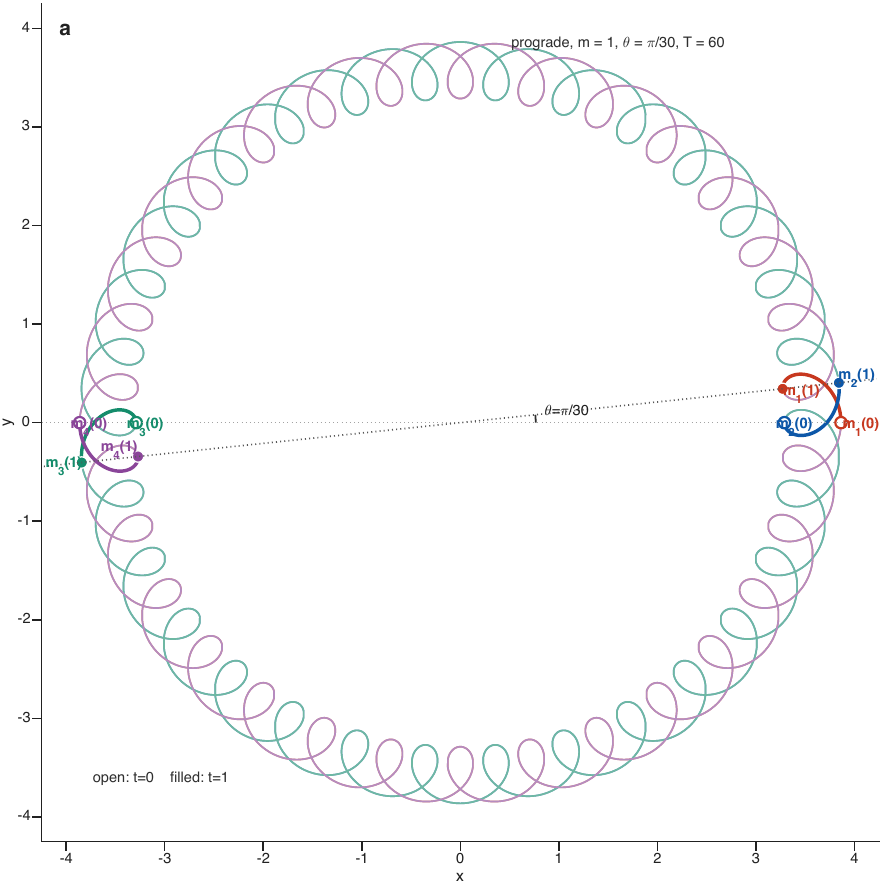}
\caption{Six sets of prograde double-double orbits with equal masses and $\theta=\frac{\pi}{4},\frac{\pi}{5},\frac{\pi}{6},\frac{\pi}{7},\frac{\pi}{20},\frac{\pi}{30}$.}
\label{fig:prograde-family}
\end{figure}
\begin{lemma}\label{lem:orbit-symmetry}
    Given a double-double orbit
\[\zeta(t) = [q_{1x}, q_{1y}, q_{2x}, q_{2y}, p_{1x}, p_{1y}, p_{2x}, p_{2y}]^{\top}\in E.
\]
It satisfies the following symmetry identities
\begin{equation}
\zeta(t + 2) = \hat{S}_f \zeta(t),  \zeta(-t + 2) = \hat{S}_r \zeta(t), \label{30}
\end{equation}
where
\[
\hat{S}_f = \begin{bmatrix} \hat{R}_f & 0 \\ 0 & \hat{R}_f \end{bmatrix}, \quad
\hat{R}_f = \begin{bmatrix} \hat{R} & 0 \\ 0 & \hat{R} \end{bmatrix}, \quad
\hat{R} = \begin{bmatrix} \cos 2\theta & -\sin 2\theta \\ \sin 2\theta & \cos 2\theta \end{bmatrix},
\]

\[
\hat{S}_r = \begin{bmatrix} \hat{R}_r & 0 \\ 0 & -\hat{R}_r \end{bmatrix}, \quad
\hat{R}_r = \begin{bmatrix} \hat{P} & 0 \\ 0 & \hat{P} \end{bmatrix}, \quad
\hat{P} = \begin{bmatrix} \cos 2\theta & \sin 2\theta \\ \sin 2\theta & -\cos 2\theta \end{bmatrix}.
\]

\end{lemma}
By \eqref{30}, we have
\begin{equation}
\zeta(t) = \hat{S}_f^{\top} \zeta(t+\frac{T}{n}),  \zeta(t) = \hat{S}_r^{\top} \zeta(\frac{T}{n}-t).
\label{eq:2-2}
\end{equation}
So we define a $\mathbb{Z}_n$-group action for the generator $g_1\in \mathbb{Z}_n$ and
a $\mathbb{Z}_2$-group action for the generator $g_2\in \mathbb{Z}_2$,where
\begin{equation*}
g_1: E \to E, \quad z(t) \mapsto \hat{S}_f^{\top} z\left(t + \frac{T}{n}\right),
\end{equation*}
\begin{equation*}
g_2: E \to E, \quad z(t) \mapsto \hat{S}_r^{\top} z\left(\frac{T}{n}-t\right).
\end{equation*}

Obviously, $\zeta(t)$ is in the dihedral group $D_n$-invariant loop space.In terms of generators and relations, $D_n$ is
\[D_n=\langle g_1,g_2\mid g_1^n=g_2^2=e, g_1g_2=g_2g_1^{-1}\rangle.\]
With a slight abuse of notation, we denote with the same symbol the dihedral group as well as its unitary representation in \(E\). Thus, the group \(D_n\) (actually a unitary representation of the dihedral group \(D_n\)) is presented as follows
\[
D_n := \langle \hat{S}_f^{\top}, \hat{S}_r^{\top} \in U(E) \mid (\hat{S}_f^{\top})^n = (\hat{S}_r^{\top})^2 = (\hat{S}_f^{\top}\hat{S}_r^{\top})^2 = I_8 \rangle \subset U(E).
\]
\section{The symplectic Jordan block decomposition of the poincare map}\label{sec:jordan}

Let $z = z(t) = \bigl(\dot{x}(t)^{\top}, x(t)^{\top}\bigr)^{\top} = \bigl(y(t)^{\top}, x(t)^{\top}\bigr)^{\top}$ be a periodic solution of the Hamiltonian system:
\begin{align}
\dot{z}(t) &= J DH(t,z(t)),  \label{31}\\
z(0) &= z(T),\label{32}
\end{align}
where
\[
J = \begin{bmatrix}
0 & -I_n \\
I_n & 0
\end{bmatrix},
\]
$I_n$ is the identity matrix on $\mathbb{R}^n$ and $H$ is the corresponding Hamiltonian.

The associated fundamental solution $\gamma \equiv \gamma(t)$ satisfies
\begin{align}
\dot{\gamma}(t) &= J D^2 H(t,z(t))\,\gamma(t),  \label{eq:3-3}\\
\gamma(0) &= I_{2n}.
\label{eq:3-4}
\end{align}

It is well-known that $\gamma$ is a path in the symplectic group \[\Sp(2n) = \{ M \in \GL(2n) \mid M^{\top} J M = J \},
\]
and the monodromy matrix (or the linear Poincaré map) corresponding to the Hamiltonian $H$ is defined to be $\gamma(T)$. Linear stability of periodic orbits in the $N$-body problem is defined as follows:

\begin{definition}\label{def:linear-stability}
    A periodic orbit in the $N$-body problem is \emph{linearly stable} (or the monodromy matrix $\gamma(T)$ is stable), if the following two conditions are satisfied:
\begin{enumerate}
\item the monodromy matrix $\gamma(T)$ can be diagonalized after reduction of center of mass, linear momentum, rotation, Hamiltonian energy and angular momentum;
\item $\gamma(T)$'s eigenvalues are all on the unit circle.
\end{enumerate}
\end{definition}

As is well known for the planar N-body problem,there are 6 integrals: 2 for linear momentum, 2 for center of mass, 1 for angular momentum, and 1 for energy. For the double-double orbits, the center of masses is fixed to the origin,the dimension of the eigenspace associated to the eigenvalue 1 is at least 4.

Let $\gamma(t)$ be the fundamental solution of the double-double orbits, and set $M=\hat{S}_f^{\top}\gamma(\frac{T}{n})$, from the $\mathbb{Z}_n$-invariance property, we have
\[\gamma(T)=M^n.\]
So the stabilities of $\gamma(T)$ and $M$ are the same. Now we analyze the eigenvalues of $M$.

Note that if \( x(t) \) is a period \( T \) solution of the Newton system, then \( h^{-2/3} x(ht) \) is also a solution with period \( T/h \). Let \( z_h(t) = (h^{1/3} \dot{x}(ht), h^{-2/3} x(ht))^{\top} \) be the double-double orbit, then \( z_h(t) \) is a solution of the corresponding Hamiltonian system. it satisfies
\begin{equation}
\hat{S}_f^{\top} z_h(T_h/n) = z_h(0),\label{36}
\end{equation}
where \( T_h = T/h \) is the period of the orbit \( z_h(t) \) and
\begin{equation}
H(z_h) = h^{2/3} H(z_1).
\label{eq:3-6}
\end{equation}
Note that for the double-double orbit as a bounded orbit in the planar four body problem, its energy is negative, which is an important fact for calculating the normal form of \( M \). The next lemma is motivated by \cite{E}(pp. 70--71), \cite{Long2002}(p. 327) in their studies of the closed characteristics in convex hypersurface for Hamiltonian system in \( \mathbb{R}^{2n} \).
\begin{lemma}\label{lem:energy-invariance}
    With the notations above and \( \gamma(t)(t \in [0, \frac{T}{n}]) \) the fundamental solution path of the Hamiltonian system, the matrix $M=\hat{S}_f^{\top}\gamma(\frac{T}{n})$ satisfies
\begin{equation}
M \dot{z}(0) = \dot{z}(0), \label{34}
\end{equation}
\begin{equation}
-\frac{T}{n} \dot{z}(0) + M \frac{\mathrm{d}}{\mathrm{d}h} z_h(0)\big|_{h=1} = \frac{\mathrm{d}}{\mathrm{d}h} z_h(0)\big|_{h=1}. \label{35}
\end{equation}
\end{lemma}
\begin{proof}

Since
\begin{equation}
    \tilde{S}_d z(t + T/n) = z(t), \label{33}
\end{equation}
differentiating this formula with respect to $t$, we get \eqref{34} by setting $t=0$. To prove \eqref{35}, from \eqref{36}, we get
\begin{equation}
    \gamma(T_h/n)\frac{\mathrm{d}}{\mathrm{d}h} z_h(0) = \frac{\mathrm{d}}{\mathrm{d}h} z_h(T_h/n). \label{37}
\end{equation}
Differentiating \eqref{36} with respect to $h$ yields
\begin{equation}
    \tilde{S}_d \dot{z}(T_h/n)\frac{d(T_h/n)}{dh} + \tilde{S}_d \frac{\mathrm{d}}{\mathrm{d}h} z_h(T_h/n) = \frac{\mathrm{d}}{\mathrm{d}h} z_h(0). \label{38}
\end{equation}
Plug \eqref{37} into \eqref{38}, and let $h=1$; we get \eqref{35}.
\end{proof}

\begin{lemma}\label{lem:energy-jordan}
     Under the same setting as the last lemma, there exist \( P \in \Sp(8) \) and \( M_1 \in \Sp(6) \), such that
\begin{equation}
M = P^{-1}(N_1(1,1) \diamond M_1)P.
\label{eq:3-12}
\end{equation}
\end{lemma}
\begin{proof}
Let \( \xi_1 = \frac{T}{n}\dot{z}(0) \) and \( \xi_2 = \frac{\mathrm{d}}{\mathrm{d}h} z_h(0)\big|_{h=1} \). Direct computations show that
\[
\omega(\xi_1, \xi_2) = \langle J T J H'(z(0)), \frac{\mathrm{d}}{\mathrm{d}h} z_h(0)\big|_{h=1} \rangle = -\frac{T}{n} \frac{\mathrm{d}}{\mathrm{d}h} H(z_h) > 0.
\]
So the space spanned by \( \xi_1, \xi_2 \) is an invariant symplectic subspace of \( M \), and furthermore they form the symplectic basis of this subspace. When \( M \) restricts to this subspace, it is \( N_1(1,1) \). Since \( M \) is a symplectic matrix, we have the result.
\end{proof}

\begin{remark}\label{rem:translation-jordan}
    The Jordan block of \( \hat{S}_f^{\top}\gamma(\frac{T}{n}) \) corresponding to the translation invariance is \( N_1(1,1) \).
\end{remark}

In the following, we will compute the normal form corresponding to the angular momentum. Let \( A_d = \begin{pmatrix} A & 0 \\ 0 & A\end{pmatrix} \) with \( A = \begin{pmatrix} \tilde{A} & 0 \\ 0 & \tilde{A} \end{pmatrix} \) and \( \tilde{A} = \begin{pmatrix} 0 & -1 \\ 1 & 0 \end{pmatrix} \). Then \( \Gamma_s = \exp(sA_d) \) is the rotation flow, it commutes with the Hamiltonian flow \( \phi_t \) of the double-double, that is
\begin{equation}
    \phi_t \Gamma_s(z_0) = \Gamma_s \phi_t(z_0).
\label{eq:3-13}
\end{equation}
Choose \( z_0 \) to be a point of the eight, set \( t = T/n \), and differentiate \eqref{eq:3-13} with respect to \( s \) at \( s = 0 \); we get
\begin{equation}
    \gamma(T/n)A_d z_0 = A_d \hat{S}_f z_0.
\label{eq:3-14}
\end{equation}
Since \( A_d \hat{S}_f= \hat{S}_f A_d \), we have
\begin{equation}
   \hat{S}_f^{\top}\gamma(T/n)A_d z_0 = A_d z_0.
\label{eq:3-15}
\end{equation}
This implies that \( 1 \) is an eigenvalue of \( M \). we have proved:
\begin{lemma}\label{lem:jordan-decomposition}
    There exist \( P \in \Sp(8) \), \( M_2 \in \Sp(4) \) such that
\begin{equation}
M = P^{-1}(N_1(1,1) \diamond N_1(1,b) \diamond M_2)P, \label{2}
\end{equation}
where \( b = 1, 0 \) or \( -1 \).
\end{lemma}

\begin{remark}\label{rem:angular-jordan}
\begin{enumerate}
\item From this lemma, \( M_2 \) is a \( 4 \times 4 \) symplectic matrix which is the essential part for the stability problem.

\item $N_1(1,b)$ is the Jordan block corresponding to the angular momentum,where \( b = 1, 0 \) or \( -1 \).Since $\gamma(T)=M^n$,it is obvious that the Jordan block of $M$ corresponding to the angular momentum is of the form $N_1(1,b)$ if and only if the Jordan block of $\gamma(T)$ corresponding to the angular momentum is $N_1(1,b)$ whether $n$ is odd or even.

\item By taking the symplectic
inner product of all eigenvectors and generalized eigenvectors corresponding to 1, we have numerically checked that $b=1$, so
\begin{equation}
M = P^{-1}(N_1(1,1) \diamond N_1(1,1) \diamond M_2)P.
\label{eq:3-17}
\end{equation}

\end{enumerate}
\end{remark}

In the following, for the double-double orbits with $\theta=\frac{\pi}{n}$, we analyze their stability by distinguishing between even and odd values of the integer $n$.

We first review the definition of Maslov index, the details can be seen in \cite{HS2009} (or \cite{CLM1994, Zhu2006}). Denote \(\operatorname{Lag}(2n)\) the set of real Lagrangian subspaces of \((\mathbb{R}^{2n},\omega)\). For a given Lagrangian subspace \(W\) of \((\mathbb{R}^{2n},\omega)\) and a path \(\Lambda(t)\) in \(\operatorname{Lag}(2n)\), the Maslov index can be defined by
\[
\mu(W,\Lambda(t),t\in[0,T]) = \bigl[e^{-\epsilon J}\Lambda(t),\, t\in[0,T] : \overline{O_1(W)}\bigr],
\]
where \(O_1(W)=\{V\mid V\in\operatorname{Lag}(2n)\text{ and }\dim(V\cap W)=1\}\), \([\cdot:\cdot]\) denotes the intersection number and \(0<\epsilon\ll 1\). For the complex case, let \(\operatorname{Lag}(\mathbb{C}^{2n})\) denote the set of complex Lagrangian subspaces in \((\mathbb{C}^{2n},\omega)\). For \(\Lambda\in\operatorname{Lag}(\mathbb{C}^{2n})\) and a path of complex Lagrangian subspaces \(V(t)\) in \(\operatorname{Lag}(\mathbb{C}^{2n})\), the Maslov index is defined by
\[
\mu(\Lambda,V(t),t\in[0,T]) = \bigl[e^{-\epsilon J}G(V(t)),\, t\in[0,T] : \Sigma_{G(V(t))}\bigr],
\]
where \(G(\cdot)\) is a homeomorphism from \(\operatorname{Lag}(\mathbb{C}^{2n})\) to the unitary group \(U(n)\) and \(\Sigma_{G(V(t))}\) is the singular cycle of \(G(V(t))\). Thus the Maslov index of the real Lagrangian subspace path is equivalent to that in the complex case (see \cite{HS2009}).

For \(M\in\Sp(2n)\), we define
\[
\Gr(M)=\{(x,Mx)\mid x\in\mathbb{R}^{2n}\}.
\]
Then the set \(\Gr(M)\) is a Lagrangian subspace of the symplectic vector space \((\mathbb{R}^{2n}\oplus\mathbb{R}^{2n}, -\omega\oplus\omega)\). Let \(z(t)\) be the solution of the retrograde double-double orbit and \(\gamma(t)\) be the associated fundamental solution. Clearly, \(\Gr(\gamma(t))\) is a path of Lagrangian subspaces. We denote by
\(
\mu\bigl(\Gr(Q^{\top}), \Gr(\gamma(t)), t\in[0,T]\bigr)
\)
 the Maslov index of a symplectic path \(\gamma(t)\) with respect to \(Q\in\Sp(2n)\).
Note that \(\hat{R}_f \) is a \(4\times 4\) orthogonal matrix, it is easily seen that \(\hat{R}_f ,\hat{S}_f \in\Sp(8)\).

We call a solution of the Newtonian system \textit{nondegenerate} if the multiplicity of the eigenvalue 1 in the Poincaré map is exactly 4. By using MATLAB, according to the eigenvalues of the $\gamma(T)$:

\begin{fact}
The double double orbit is nondegenerate.
\end{fact}

More precisely, the multiplicity of the eigenvalue -1 in the monodromy matrix $\gamma(T)$ is exactly 0. Recall that $\gamma(2T)=\gamma(T)^2$, we obtain that the multiplicity of the eigenvalue 1 in the monodromy matrix $\gamma(2T)$ is exactly 4.

\section{Morse index and Maslov-type index}\label{sec:common-indices}
The retrograde and prograde double-double orbits have the same symmetry matrices. In this section, we give the index formulas used for both types of orbits. The different Morse indexs will be substituted in Section~\ref{sec:stability}.

\subsection{The Morse index and the generalized Bott-type iteration formula}\label{sec:common-bott}
We use the notation $m_1,m_2,m_3,m_4$ introduced in Section~\ref{sec:introduction}. When $n$ is even, let $\mathbb{Z}_n,\mathbb{Z}_2,\mathbb{Z}_{n/2}$ denote the groups generated by $g_1,g_1^{\frac{n}{2}}$ and $g_1^2$, respectively. The Morse indexs $m_1,m_2,m_3,m_4$ correspond to the total loop space and its $\mathbb{Z}_2$-, $\mathbb{Z}_{n/2}$- and $\mathbb{Z}_n$-invariant loop subspaces. When $n$ is odd, we use $m_1,m_3,m_4$, where $m_3$ is computed on $[0,4]$ with the symmetry boundary condition, as in Appendix~\ref{app:numerical-index}.

For simplifying notations, we define
\begin{equation}
\mu(\omega):=\mu(\Gr(\omega(\hat{S}_f^{\top})^{\top}),\Gr(\gamma(t)),t\in[0,T/n]),\quad \forall\omega\in\mathbf{U}.
\label{eq:common-mu}
\end{equation}

\begin{lemma}\label{lem:common-indices}
For either a retrograde or a prograde double-double orbit with $3\leq n\leq30$, the fundamental solution $\gamma$ satisfies
\begin{align}
\mu(\Gr((\hat{S}_f^{\top})^{\top}),\Gr(\gamma(t)),t\in[0,T/n])&=m_4,\label{eq:4-1}\\
\mu(\Gr(((\hat{S}_f^{\top})^2)^{\top}),\Gr(\gamma(t)),t\in[0,T/{\frac{n}{2}}])&=m_3,\label{eq:4-2}\\
\mu(\Delta,\Gr(\gamma(t)),t\in[0,T])&=m_1+4.\label{eq:4-4}
\end{align}
If $n$ is even, we also have
\begin{equation}
\mu(\Gr(((\hat{S}_f^{\top})^{\frac{n}{2}})^{\top}),\Gr(\gamma(t)),t\in[0,T/2])=m_2.
\label{eq:4-3}
\end{equation}
\end{lemma}
\begin{proof}
For $n\geq3$, both $\ker(\hat{R}_f^{\top}-I_4)$ and $\ker((\hat{R}_f^{\top})^2-I_4)$ are zero. If $n$ is even, $\ker((\hat{R}_f^{\top})^{\frac{n}{2}}-I_4)$ is also zero. The formulas follow from Theorem 1.2 of \cite{HS2009}, or Theorem~\ref{thm:morse_maslov} in Appendix~\ref{app:index-theory}.
\end{proof}

By the generalized Bott-type iteration formula in Theorem~\ref{thm:bott}, Lemma~\ref{lem:common-indices} gives
\begin{align}
&\mu(1)=m_4,\quad \mu(1)+\mu(-1)=m_3,\label{eq:common-index-values}\\
&\sum_{j=1}^{n}\mu(\exp(2\pi j\sqrt{-1}/n))=m_1+4.\label{eq:common-full-bott}
\end{align}
Since $\mu(\omega)=\mu(\bar{\omega})$, when $n$ is even,
\begin{equation}
\sum_{j=1}^{\frac{n-2}{2}}\mu(\exp(2\pi j\sqrt{-1}/n))=\frac{1}{2}(m_1+4-m_3),
\label{eq:4-11}
\end{equation}
and
\begin{equation}
\sum_{j=1}^{\frac{n}{2}}\mu(\exp(4\pi j\sqrt{-1}/n))=m_2.
\label{eq:common-even-bott}
\end{equation}
When $n$ is odd, we have
\begin{equation}
\sum_{j=1}^{\frac{n-1}{2}}\mu(\exp(2\pi j\sqrt{-1}/n))=\frac{1}{2}(m_1+4-m_4).
\label{eq:common-odd-sum}
\end{equation}
The case $n=2$ will be treated separately in Section~\ref{sec:stability}.

\subsection{Splitting numbers and the index formula}\label{sec:common-splitting}
In the following, we will compute \(D_\omega(\hat{S}_f^{\top})\), where
\begin{equation}
D_\omega(M)=S_M^+(1)+\sum_{\omega_0}(S_M^+(\omega_0)-S_M^-(\omega_0))-S_M^-(\omega),    \label{1}
\end{equation}
and the sum runs over all the eigenvalues $\omega_0$ of $M$ belonging to the part of $\mathbf{U}^{+}$ or $\mathbf{U}^{-}$.  We set \(\mathbf{U}^{+} = \{\omega = \exp(\sqrt{-1}\theta) \in \mathbf{U} \mid \theta \in [0, \pi]\}\).

\begin{lemma}\label{lem:common-splitting}
For the double-double orbits with $\theta=\frac{\pi}{n}$,
 \begin{enumerate}[label=(\arabic*)]
 \item When $n\geq3$
 \begin{equation}
      D_\omega(\hat{S}_f^{\top}) = 0, \quad \omega \in\mathbf{U}^{+}\setminus\{\exp(2\pi\sqrt{-1}/n)\},
\label{eq:4-13}
\end{equation}
and
\begin{equation}
    D_{\exp(2\pi\sqrt{-1}/n)}(\hat{S}_f^{\top}) = -2 .
\label{eq:4-14}
\end{equation}
\item When $n=2$
\begin{equation}
      D_\omega(\hat{S}_f^{\top}) = 0, \quad \omega \in\mathbf{U}^{+}\setminus\{-1\},
\label{eq:4-15}
\end{equation}
and
\begin{equation}
    D_{-1}(\hat{S}_f^{\top}) = -4 .
    \label{eq:4-16}
\end{equation}
\end{enumerate}
\end{lemma}

\begin{proof}
We only prove (1) , (2) is similar.
Note that \(\hat{S}_f^{\top} = \hat{R}_f^{\top} \diamond \hat{R}_f^{\top}\), and
\[\hat{R}_f^{\top}=P^{-1}(\hat{R}^{\top} \diamond \hat{R'}^{\top})P.\]
where $\hat{R'} = \begin{bmatrix} \cos 2(\pi-\theta) & -\sin 2(\pi-\theta) \\ \sin 2(\pi-\theta) & \cos 2(\pi-\theta) \end{bmatrix}$
then we have
\begin{align*}
(S_{\hat{R}_f^{\top}}^+(e^{\frac{2\pi}{n}\sqrt{-1}}), S_{\hat{R}_f^{\top}}^-(e^{\frac{2\pi}{n}\sqrt{-1}}))
&= (S_{\hat{R}^{\top}}^+(e^{\frac{2\pi}{n}\sqrt{-1}}), S_{\hat{R}^{\top}}^-(e^{\frac{2\pi}{n}\sqrt{-1}})) \\
&\quad + (S_{\hat{R'}^{\top}}^+(e^{\frac{2\pi}{n}\sqrt{-1}}), S_{\hat{R'}^{\top}}^-(e^{\frac{2\pi}{n}\sqrt{-1}})) \\
&= (0, 1) + (1,0) \\
&= (1, 1).
\end{align*}
Since the eigenvalues of $\hat{R}_f^{\top}$ are $e^{\pm\frac{2\pi}{n}\sqrt{-1}}$,
\begin{equation}
 (S_{\hat{R}_f^{\top}}^+(\omega), S_{\hat{R}_f^{\top}}^-(\omega))=(0,0), \omega\neq e^{\frac{2\pi}{n}\sqrt{-1}}, \omega\in \mathbf{U}^{+}
\label{eq:4-17}
\end{equation}
The lemma follows from \eqref{1}.
\end{proof}

We choose any path $\xi$ in $\Sp(2n)$ connecting $I_{2n}$
to $\hat{S}_f^{\top}$. Let $\tilde{\gamma}$ be the path $\hat{S}_f^{\top}\gamma(t)$
with $t\in [0,\frac{T}{n}]$, where $\gamma(t)$ is the fundamental solution of the double-double orbit. let $\tilde{\gamma}*\xi$
be the concatenation of $\xi$ and $\tilde{\gamma}$. From \cite{HS2009}
(4.21), for any $\omega\in \mathbf{U}$,
\begin{equation}
    \mu(\omega)=\mu(\Gr(\omega),\Gr(\tilde{\gamma}))=i_{\omega}(\tilde{\gamma}*\xi)-i_{\omega}(\xi).
\label{eq:4-20}
\end{equation}
So
\begin{align}
    \mu(\omega)-\mu(1)&=i_{\omega}(\tilde{\gamma}*\xi)-i_1(\tilde{\gamma}*\xi)-(i_{\omega}(\xi)-i_1(\xi))\label{eq:4-21}\\
    &=i_{\omega}(\tilde{\gamma}*\xi)-i_1(\tilde{\gamma}*\xi)-D_\omega(\hat{S}_f^{\top}).
\label{eq:4-22}
\end{align}
From \cite{HS2009}(4.17), for $\omega\in \mathbf{U}^{+}$,
\begin{equation}
    \mu(\omega)=\mu(1)+S_M^+(1)+\sum_{\omega_0}(S_M^+(\omega_0)-S_M^-(\omega_0))-S_M^-(\omega)-D_\omega(\hat{S}_f^{\top}).\label{3}
\end{equation}
where the sum runs over all the eigenvalues $\omega_0$ of $M$ belonging to the part of $\mathbf{U}^{+}$ strictly located between 1 and $\omega$.

From \eqref{2}, we know $M = P^{-1}(N_1(1,1) \diamond N_1(1,b) \diamond M_2)P$, for some $P\in \Sp(8)$. So
\begin{equation}
    \frac{e(M)}{2}=\frac{1}{2}(e(N_1(1,1))+e(N_1(1,b))+e(M_2))=2+\frac{e(M_2)}{2}.
\label{eq:4-18}
\end{equation}
Suppose $\omega_i, i=1,\ldots,k$, are all the eigenvalues of $M_2$(same as $M$) belonging to $\mathbf{U}^{+}$ strictly located between 1 and -1, By (4.30) of \cite{HS2009},
\begin{equation}
    \frac{e(M_2)}{2}\geq S_{M_2}^+(1)+\sum_{\omega_0}(S_{M_2}^+(\omega_0)+S_{M_2}^-(\omega_0))+S_{M_2}^-(-1).
\label{eq:4-19}
\end{equation}

\begin{proposition}\label{prop:common-stability}
Suppose that $3\leq n\leq30$, the non-degeneracy assumption (2) of Theorems~\ref{thm:retrograde-even}--\ref{thm:prograde-odd} holds, and $b=1$ in \eqref{2}.
\begin{enumerate}
\item If $m_4=0$ and $m_3=2$, then the orbit is either spectrally stable or hyperbolic.
\item If $m_4=2$ and $m_3=4$, then the orbit is linearly stable. Moreover, there exist $P\in\Sp(8)$ and $\theta_1,\theta_2\in(0,\pi)$ such that
\begin{equation}
M=P^{-1}(N_1(1,1)\diamond N_1(1,1)\diamond R(\theta_1)\diamond R(\theta_2))P.
\label{eq:common-prograde-normal-form}
\end{equation}
\end{enumerate}
\end{proposition}
\begin{proof}
The non-degeneracy assumption and $\gamma(2T)=M^{2n}$ imply that $M_2$ has no eigenvalues at $1$ or $-1$. Since $b=1$, \eqref{2} gives $S_M^+(1)=2$ and $S_M^-(-1)=0$. Substituting $\omega=-1$ in \eqref{3} and using \eqref{eq:common-index-values} and Lemma~\ref{lem:common-splitting}, we obtain
\begin{equation}
\sum_{\omega_0}(S_{M_2}^+(\omega_0)-S_{M_2}^-(\omega_0))=m_3-2m_4-2,
\label{eq:common-splitting-sum}
\end{equation}
where the sum runs over the eigenvalues of $M_2$ in $\mathbf{U}^{+}$ strictly located between $1$ and $-1$.

For (1), the right-hand side of \eqref{eq:common-splitting-sum} is zero. If exactly two eigenvalues of $M_2$ lie on the unit circle, counted with algebraic multiplicity, they form a simple conjugate pair. The difference of the splitting numbers of its eigenvalue in $\mathbf{U}^{+}$ is $1$ or $-1$, which contradicts \eqref{eq:common-splitting-sum}. Since $M_2\in\Sp(4)$, either none or all four of its eigenvalues lie on the unit circle. This gives (1).

For (2), \eqref{eq:common-splitting-sum} gives
\begin{equation}
\sum_{\omega_0}(S_{M_2}^+(\omega_0)-S_{M_2}^-(\omega_0))=-2.
\label{eq:common-prograde-sum}
\end{equation}
By \eqref{eq:4-19},
\[
2\leq\sum_{\omega_0}(S_{M_2}^+(\omega_0)+S_{M_2}^-(\omega_0))\leq\frac{e(M_2)}{2}\leq2.
\]
Here $e(M_2)$ is the total algebraic multiplicity of its eigenvalues on the unit circle. Thus all four eigenvalues lie on the unit circle, and
\[
\sum_{\omega_0}S_{M_2}^+(\omega_0)=0,\qquad \sum_{\omega_0}S_{M_2}^-(\omega_0)=2.
\]
The basic normal forms and splitting numbers in Appendix~\ref{app:index-theory} give \eqref{eq:common-prograde-normal-form}. Hence the orbit is linearly stable.
\end{proof}

For $3\leq n\leq30$, under the same non-degeneracy assumption, $M_2$ has no eigenvalues at $e^{\frac{2\pi j\sqrt{-1}}{n}}$: otherwise $\gamma(2T)=M^{2n}$ would have eigenvalue $1$ with multiplicity greater than $4$. Consequently, \eqref{3} gives the following formulas for $b=1$:
\begin{equation}
\mu(e^{\frac{2\pi\sqrt{-1}}{n}})=m_4+4+\sum_{\omega_0}(S_{M_2}^+(\omega_0)-S_{M_2}^-(\omega_0)),
\label{eq:common-first-sample}
\end{equation}
where the sum runs over the eigenvalues strictly located between $1$ and $e^{\frac{2\pi\sqrt{-1}}{n}}$ in $\mathbf{U}^{+}$, and
\begin{equation}
\mu(e^{\frac{2\pi j\sqrt{-1}}{n}})=m_4+2+\sum_{\omega_0}(S_{M_2}^+(\omega_0)-S_{M_2}^-(\omega_0)),\qquad 2\leq j<\frac{n}{2},
\label{eq:common-other-samples}
\end{equation}
where the sum runs over the eigenvalues strictly located between $1$ and $e^{\frac{2\pi j\sqrt{-1}}{n}}$ in $\mathbf{U}^{+}$. We will use these two formulas for both types of double-double orbits.

\section{Stability of the double-double orbits}\label{sec:stability}
In this section, we substitute the Morse indexs of the retrograde and prograde double-double orbits into the formulas of Section~\ref{sec:common-indices}. In the proofs of Theorems~\ref{thm:retrograde-even}--\ref{thm:prograde-odd}, assumptions (2) and (3) are understood. In particular, $b=1$ by Remark~\ref{rem:angular-jordan}.

\subsection{Retrograde double-double orbits}\label{sec:retrograde}
Using MATLAB, we get $m_1=2n-4$, $m_3=2$ and $m_4=0$ for $3\leq n\leq30$. When $n$ is even, we also have $m_2=n-2$. We give one proof for both even and odd $n$, under the corresponding assumptions of Theorems~\ref{thm:retrograde-even} and \ref{thm:retrograde-odd}.

\begin{proof}[Proof of Theorems~\ref{thm:retrograde-even} and \ref{thm:retrograde-odd}]
By \eqref{eq:common-index-values}, $\mu(1)=0$ and $\mu(-1)=2$. Substituting the Morse indexs into \eqref{eq:4-11} and \eqref{eq:common-odd-sum}, we obtain, for both even and odd $n$,
\begin{equation}
\sum_{j=1}^{\lfloor\frac{n-1}{2}\rfloor}\mu(e^{\frac{2\pi j\sqrt{-1}}{n}})
=2\left\lfloor\frac{n-1}{2}\right\rfloor+1.
\label{eq:retrograde-half-sum}
\end{equation}
The right-hand side is odd.

By assumption (3) and Remark~\ref{rem:angular-jordan}, $b=1$. Proposition~\ref{prop:common-stability}(1) shows that $M_2$ is either spectrally stable or hyperbolic. Moreover, \eqref{eq:common-splitting-sum} gives
\begin{equation}
\sum_{\omega_0}\bigl(S_{M_2}^+(\omega_0)-S_{M_2}^-(\omega_0)\bigr)=0,
\label{eq:retrograde-splitting-sum}
\end{equation}
where the sum runs over the eigenvalues of $M_2$ strictly located between $1$ and $-1$ in $\mathbf{U}^{+}$.

If $M_2$ is hyperbolic, the sums in \eqref{eq:common-first-sample} and \eqref{eq:common-other-samples} are zero. Hence the index values in \eqref{eq:retrograde-half-sum} are $4,2,\ldots,2$, whose sum is even. This contradicts \eqref{eq:retrograde-half-sum}.

If $M_2$ has only one distinct eigenvalue in this part of $\mathbf{U}^{+}$, its algebraic multiplicity is two. By \eqref{eq:retrograde-splitting-sum}, the difference of its splitting numbers is zero. The same formulas again give the index values $4,2,\ldots,2$, leading to the same contradiction.

Thus $M_2$ has two distinct eigenvalues in this part of $\mathbf{U}^{+}$ and their conjugates. All four eigenvalues are simple and lie on the unit circle, so $M_2$ can be diagonalized. This proves the linear stability in both the even and odd cases.
\end{proof}

\begin{corollary}\label{cor:retrograde-normal-form}
Under the assumptions of Theorem~\ref{thm:retrograde-even} for even $n$, or Theorem~\ref{thm:retrograde-odd} for odd $n$, there exist $P\in\Sp(8)$ and $0<\theta_1<\theta_2<\pi$ such that
\begin{equation}
M=P^{-1}\bigl(N_1(1,1)\diamond N_1(1,1)\diamond R(\theta_1)\diamond R(2\pi-\theta_2)\bigr)P.
\label{eq:retrograde-normal-form}
\end{equation}
There is a unique integer $j$, $1\leq j\leq\lfloor\frac{n-1}{2}\rfloor$, such that
\begin{equation}
\frac{2(j-1)\pi}{n}<\theta_1<\frac{2j\pi}{n}<\theta_2<\min\left\{\frac{2(j+1)\pi}{n},\pi\right\}.
\label{eq:retrograde-angle-intervals}
\end{equation}
When $n$ is even, this integer $j$ is odd.
\end{corollary}
\begin{proof}
By the preceding proof, write the two eigenvalues in the upper half of the unit circle as $e^{\theta_1\sqrt{-1}}$ and $e^{\theta_2\sqrt{-1}}$, where $0<\theta_1<\theta_2<\pi$. They are simple, so the differences of their splitting numbers are $1$ or $-1$. By \eqref{eq:retrograde-splitting-sum}, these differences are opposite.

Without the sums involving splitting numbers, the sum of the index values in \eqref{eq:common-first-sample} and \eqref{eq:common-other-samples} would be
\[
4+2\left(\left\lfloor\frac{n-1}{2}\right\rfloor-1\right)
=2\left\lfloor\frac{n-1}{2}\right\rfloor+2.
\]
The actual sum in \eqref{eq:retrograde-half-sum} is less by one. Among the angles $\frac{2k\pi}{n}$, $1\leq k\leq\lfloor\frac{n-1}{2}\rfloor$, each angle strictly between $\theta_1$ and $\theta_2$ changes the index by the difference of the splitting numbers at $e^{\theta_1\sqrt{-1}}$; at all other angles the index is unchanged. It follows that this difference is $-1$ and exactly one of these angles lies between $\theta_1$ and $\theta_2$. Therefore,
\begin{align*}
\bigl(S_{M_2}^+(e^{\theta_1\sqrt{-1}}),S_{M_2}^-(e^{\theta_1\sqrt{-1}})\bigr)&=(0,1),\\
\bigl(S_{M_2}^+(e^{\theta_2\sqrt{-1}}),S_{M_2}^-(e^{\theta_2\sqrt{-1}})\bigr)&=(1,0).
\end{align*}
The basic normal forms give \eqref{eq:retrograde-normal-form}. Assumption (2) excludes eigenvalues at the angles $\frac{2k\pi}{n}$, so all inequalities in \eqref{eq:retrograde-angle-intervals} are strict.

When $n$ is even, \eqref{eq:common-even-bott} gives
\begin{equation}
\sum_{k=1}^{\frac{n}{2}}\mu(e^{\frac{4\pi k\sqrt{-1}}{n}})=m_2=n-2.
\label{eq:retrograde-even-index}
\end{equation}
This sum uses the angles $\frac{2k\pi}{n}$ with even indices $k$. If $j$ in \eqref{eq:retrograde-angle-intervals} were even, both $e^{\frac{2j\pi\sqrt{-1}}{n}}$ and its conjugate would occur in \eqref{eq:retrograde-even-index}. Each of their index values would be less by one. Since $\mu(1)=0$ and all other values in this sum would be $2$, the sum would be $n-4$, contradicting \eqref{eq:retrograde-even-index}. Thus $j$ is odd.
\end{proof}

\begin{remark}\label{rem:retrograde-numerics}
Using MATLAB, for $3\leq n\leq30$, we check that the 8 eigenvalues of $M$ are $1,1,1,1,\omega_1,\bar{\omega}_1,\omega_2,\bar{\omega}_2$, where $\omega_1=e^{i\theta_1}$ and $\omega_2=e^{i\theta_2}$. The numerical results give $j=1$ in \eqref{eq:retrograde-angle-intervals}, namely
\[
\theta_1\in\left(0,\frac{2\pi}{n}\right),\qquad
\theta_2\in\left(\frac{2\pi}{n},\min\left\{\frac{4\pi}{n},\pi\right\}\right).
\]
These results are well in accord with the numerical results of \cite{Peng2015}. The index argument above allows the other intervals in \eqref{eq:retrograde-angle-intervals}; the choice $j=1$ here comes from the numerical calculation.
\end{remark}

\begin{remark}\label{rem:retrograde-n2}
The case $n=2$ is rather special. Suppose that $m_1=m_3=2$, $m_2=m_4=0$, together with assumptions (2) and (3) of Theorem~\ref{thm:retrograde-even}. By Theorem 1.2 of \cite{HS2009} and the generalized Bott-type iteration formula, we have
\begin{equation}
\mu(1)=0,\qquad \mu(1)+\mu(-1)=m_1+4=6.
\label{eq:retrograde-n2-indices}
\end{equation}
Since $D_{-1}(\hat{S}_f^{\top})=-4$ by Lemma~\ref{lem:common-splitting}, \eqref{3} gives
\begin{equation}
6=\mu(-1)=2+\sum_{\omega_0}\bigl(S_{M_2}^+(\omega_0)-S_{M_2}^-(\omega_0)\bigr)+4.
\label{eq:retrograde-n2-splitting}
\end{equation}
Thus the sum in \eqref{eq:retrograde-n2-splitting} is zero. By the argument of Proposition~\ref{prop:common-stability}(1), we can only obtain spectral stability or hyperbolicity. The odd sum in \eqref{eq:retrograde-half-sum} is not available for $n=2$.

Using MATLAB, we check that the 8 eigenvalues of $M$ are $1,1,1,1,\omega_1,\bar{\omega}_1,\omega_2,\bar{\omega}_2$, with $\omega_1=e^{i\theta_1}$, $\omega_2=e^{i\theta_2}$ and $\theta_1,\theta_2\in(0,\pi)$, in accord with \cite{Peng2015}.
\end{remark}

\subsection{Prograde double-double orbits}\label{sec:prograde}
\subsubsection{\texorpdfstring{$4\leq n\leq30$}{ 4 <= n <= 30}}\label{sec:prograde-big-n}
Using MATLAB, we get $m_1=2n$, $m_3=4$ and $m_4=2$ for $5\leq n\leq30$. When $n$ is even, we also have $m_2=n$. For $n=4$, the values are $m_1=10$, $m_2=m_3=4$ and $m_4=2$. Thus Proposition~\ref{prop:common-stability}(2) applies to both the even and odd cases. It remains to determine the distributions of the eigenvalues on the unit circle.

\begin{proof}[Proof of Theorems~\ref{thm:prograde-even} and \ref{thm:prograde-odd}]
\leavevmode\par
\noindent\textbf{Case 1.}\quad $5\leq n\leq30$.

For even $n$, \eqref{eq:4-11} gives
\begin{equation}
\sum_{j=1}^{\frac{n-2}{2}}\mu(\exp(2\pi j\sqrt{-1}/n))=n.
\label{eq:prograde-even-sum}
\end{equation}
For odd $n$, \eqref{eq:common-odd-sum} gives
\begin{equation}
\sum_{j=1}^{\frac{n-1}{2}}\mu(\exp(2\pi j\sqrt{-1}/n))=n+1.
\label{eq:prograde-odd-sum}
\end{equation}
By the proof of Proposition~\ref{prop:common-stability}(2), all splitting numbers $S_{M_2}^+(\omega_0)$ in $\mathbf{U}^{+}$ are zero and the sum of $S_{M_2}^-(\omega_0)$ is $2$. Substituting $m_4=2$ in \eqref{eq:common-first-sample} and \eqref{eq:common-other-samples}, we have
\begin{equation}
4\leq\mu(e^{\frac{2\pi\sqrt{-1}}{n}})\leq6,\qquad
2\leq\mu(e^{\frac{2\pi j\sqrt{-1}}{n}})\leq4,\quad 2\leq j<\frac{n}{2}.
\label{eq:prograde-index-bounds}
\end{equation}
The sum of these lower bounds is
\[
4+2\left(\frac{n-2}{2}-1\right)=n
\]
when $n$ is even, and
\[
4+2\left(\frac{n-1}{2}-1\right)=n+1
\]
when $n$ is odd. These are exactly the right-hand sides of \eqref{eq:prograde-even-sum} and \eqref{eq:prograde-odd-sum}. Therefore,
\begin{equation}
\mu(e^{\frac{2\pi\sqrt{-1}}{n}})=4,\qquad
\mu(e^{\frac{2\pi j\sqrt{-1}}{n}})=2,\quad 2\leq j<\frac{n}{2}.
\label{eq:prograde-index-values}
\end{equation}
By \eqref{eq:common-first-sample}, both conjugate pairs of eigenvalues of $M_2$ have their eigenvalues in $\mathbf{U}^{+}$ strictly between $1$ and $e^{\frac{2\pi\sqrt{-1}}{n}}$. Hence the angles in \eqref{eq:common-prograde-normal-form} satisfy
\[
\theta_1,\theta_2\in\left(0,\frac{2\pi}{n}\right).
\]

\noindent\textbf{Case 2.}\quad $n=4$.

Here $m_1=10$, $m_2=m_3=4$ and $m_4=2$. By \eqref{eq:common-index-values} and \eqref{eq:common-full-bott},
\begin{align}
\mu(1)=2,\qquad\mu(-1)=2,\label{eq:prograde-n4-endpoints}\\
\mu(e^{\frac{2\pi\sqrt{-1}}{4}})+\mu(-1)+\mu(e^{\frac{6\pi\sqrt{-1}}{4}})+\mu(1)&=14.\label{eq:prograde-n4-bott}
\end{align}
Since $\mu(\omega)=\mu(\bar{\omega})$, this gives
\begin{equation}
\mu(e^{\frac{2\pi\sqrt{-1}}{4}})=\mu(e^{\frac{6\pi\sqrt{-1}}{4}})=5.
\label{eq:prograde-n4-index}
\end{equation}
By \eqref{eq:common-first-sample}, exactly one of the two conjugate pairs has its eigenvalue in $\mathbf{U}^{+}$ between $1$ and $e^{\frac{2\pi\sqrt{-1}}{4}}$. Thus
\[
\theta_1\in\left(0,\frac{2\pi}{4}\right),\qquad
\theta_2\in\left(\frac{2\pi}{4},\pi\right).
\]
Both cases are linearly stable by Proposition~\ref{prop:common-stability}(2).
\end{proof}

\begin{remark}\label{rem:prograde-numerics}
Using MATLAB, we check that the 8 eigenvalues of $M$ are $1,1,1,1,\omega_1,\bar{\omega}_1,\omega_2,\bar{\omega}_2$ with $\omega_1=e^{i\theta_1}$ and $\omega_2=e^{i\theta_2}$. For $5\leq n\leq30$, both angles satisfy $\theta_1,\theta_2\in(0,\frac{2\pi}{n})$. For $n=4$, they satisfy $\theta_1\in(0,\frac{2\pi}{4})$ and $\theta_2\in(\frac{2\pi}{4},\pi)$. These results are well in accord with the numerical results of \cite{Peng2015}.
\end{remark}

\begin{corollary}\label{cor:prograde-normal-form}
Under the assumptions of Theorem~\ref{thm:prograde-even} for even $n$, or Theorem~\ref{thm:prograde-odd} for odd $n$, there exist $P\in\Sp(8)$ and $\theta_1,\theta_2$ such that
\begin{equation}
M=P^{-1}(N_1(1,1)\diamond N_1(1,1)\diamond R(\theta_1)\diamond R(\theta_2))P.
\label{eq:prograde-normal-form}
\end{equation}
For $5\leq n\leq30$, $\theta_1,\theta_2\in(0,\frac{2\pi}{n})$. For $n=4$, $\theta_1\in(0,\frac{2\pi}{4})$ and $\theta_2\in(\frac{2\pi}{4},\pi)$.
\end{corollary}

\subsubsection{\texorpdfstring{$n=3$ and $n=2$}{n=3 and n=2}}\label{sec:prograde-small-n}
In this subsection, we will further consider the case of equal-mass prograde double double orbits which has been found numerically but whose existence has not been rigorously verified, i.e. $n=3$ and $n=2$.
The following calculations are made under the non-degeneracy assumption and $b=1$, as in the preceding proofs.

\paragraph{$n=3$.} Using the notation of Section~\ref{sec:common-bott}, we get the Morse indexs in the various loop spaces by MATLAB, and they are
\(m_1=8, \quad m_3=4,\quad m_4=2.\) From \eqref{eq:common-index-values} and \eqref{eq:common-full-bott}, we have
\begin{align}
    \mu(e^{\frac{2\pi\sqrt{-1}}{3}})+\mu(e^{\frac{4\pi\sqrt{-1}}{3}})+\mu(1)&=12,\label{eq:5-67}\\
    \mu(-1)+\mu(1)&=4,\label{eq:5-68}\\
    \mu(1)&=2.
\label{eq:5-69}
\end{align}
By Lemma~\ref{lem:common-splitting} and \eqref{3},
\begin{align}
    5= \mu(e^{\frac{2\pi\sqrt{-1}}{3}})&=2+S_M^+(1)+\sum_{\omega_0}(S_M^+(\omega_0)-S_M^-(\omega_0))-S_M^-(-1)+2.\label{eq:5-70}\\
    &=6+\sum_{\omega_0}(S_M^+(\omega_0)-S_M^-(\omega_0))
\label{eq:5-71}
\end{align}
where the sum runs over all the eigenvalues $\omega_0$ of $M$ belonging to the part of $\mathbf{U^{+}}$ strictly located between 1 and $e^{\frac{2\pi \sqrt{-1}}{3}}$.
and
\begin{align}
    2= \mu(-1)&=2+S_M^+(1)+\sum_{\omega_0}(S_M^+(\omega_0)-S_M^-(\omega_0))-S_M^-(-1)+0.\label{eq:5-72}\\
    &=4+\sum_{\omega_0}(S_M^+(\omega_0)-S_M^-(\omega_0))
\label{eq:5-73}
\end{align}
where the sum runs over all the eigenvalues $\omega_0$ of $M$ belonging to the part of $\mathbf{U^{+}}$ strictly located between 1 and -1.

So there exists $\omega_1=e^{\theta_1\sqrt{-1}}$ with $\theta_1\in(0,\frac{2\pi}{3})$ and $\omega_2=e^{\theta_2\sqrt{-1}}$ with $\theta_2\in (\frac{2\pi}{3},\pi)$ such that
\begin{align}
    (S_{M_2}^+(\omega_1),S_{M_2}^-(\omega_1))&=(0,1);\label{eq:5-74}\\
    (S_{M_2}^+(\omega_2),S_{M_2}^-(\omega_2))&=(0,1).
\label{eq:5-75}
\end{align}
Base on above, $M$ must be linearly stable. We can also know
 $\exists P\in \Sp(8)$, $\theta_1\in (0,\frac{2\pi}{3})$, $\theta_2\in (\frac{2\pi}{3},\pi)$, such that
\begin{equation}
    M=P^{-1}(N_1(1,1)\diamond N_1(1,1)\diamond R(\theta_1)\diamond R(\theta_2))P.
\label{eq:5-76}
\end{equation}
\paragraph{$n=2$.} Using the notation of Section~\ref{sec:common-bott}, we get the Morse indexs in the various loop spaces by MATLAB, and they are
\(m_1=m_3=5, \quad m_2=m_4=1.\) Similarly, we have
\begin{align}
    \mu(-1)+\mu(1)&=m_1+4=9,\label{eq:5-77}\\
    \mu(1)&=1.
\label{eq:5-78}
\end{align}
By Lemma~\ref{lem:common-splitting} and \eqref{3},
\begin{align}
    8= \mu(-1)&=1+S_M^+(1)+\sum_{\omega_0}(S_M^+(\omega_0)-S_M^-(\omega_0))-S_M^-(-1)+4.\label{eq:5-79}\\
    &=7+\sum_{\omega_0}(S_M^+(\omega_0)-S_M^-(\omega_0))
\label{eq:5-80}
\end{align}
where the sum runs over all the eigenvalues $\omega_0$ of $M$ belonging to the part of $\mathbf{U^{+}}$ strictly located between 1 and -1.

So there exists $\omega_1=e^{\theta_1\sqrt{-1}}$ with $\theta_1\in(0,\pi)$ and $\lambda\in \mathbb{R}, |\lambda|>1$ such that
\begin{align}
    (S_{M_2}^+(\omega_1),S_{M_2}^-(\omega_1))&=(1,0);\label{eq:5-81}\\
    (S_{M_2}^+(\lambda),S_{M_2}^-(\lambda))&=(0,0).
\label{eq:5-82}
\end{align}
Base on above, $M$ must be linearly unstable. We can also know
 $\exists P\in \Sp(8)$, $\theta_1\in (0,\pi)$, $\lambda\in \mathbb{R}, |\lambda|>1$, such that
\begin{equation}
    M=P^{-1}(N_1(1,1)\diamond N_1(1,1)\diamond R(2\pi-\theta_1)\diamond D(\lambda))P.
\label{eq:5-83}
\end{equation}
\begin{remark}\label{rem:prograde-small-n-numerics}
    Using MATLAB, we also check that the 8 eigenvalues of $M$ with $n=3$ and $n=2$, which are well in accord with the numerical results of \cite{Peng2015}.
\end{remark}

\appendix
\section{Maslov-type index theory}\label{app:index-theory}
In this section, we briefly review the Maslov-type index theory for symplectic paths. The details can be found in \cite{Long2002}. A generalized Bott-type iteration formula for Maslov-type index in \cite{HS2009} is also reviewed here.

The symplectic group $\Sp(2n)$ is defined by
\[
\Sp(2n)=\{M\in \GL(\mathbb{R}^{2n})\mid M^{\top}JM=J\},
\]
where
\[
J=\begin{bmatrix}
0 & -I_n\\
I_n & 0
\end{bmatrix},
\]
$I_n$ is the identity matrix on $\mathbb{R}^n$, and $M^{\top}$ denotes the transpose of $M$. For $T>0$, we define the set of symplectic matrix paths by
\[
\mathcal{P}_T(2n)=\{\gamma\in C([0,T],\Sp(2n))\mid \gamma(0)=I_{2n}\}.
\]

Let $U$ be the unit circle in the complex plane. The following real function was introduced in \cite{Long2002}:
\[
D_\omega M=(-1)^{n-1}\overline{\omega}^{n}\det(M-\omega I_{2n}),\qquad \forall \omega\in U,\; M\in\Sp(2n).
\]
Then for any $\omega\in U$, the codimension one hypersurface in $\Sp(2n)$ is defined by
\[
\Sp(2n)_\omega^0=\{M\in\Sp(2n)\mid D_\omega(M)=0\}.
\]
For any $M\in\Sp(2n)_\omega^0$, we define a co-orientation of $\Sp(2n)_\omega^0$ at $M$ by the positive direction $\frac{d}{dt}Me^{tJ}\big|_{t=0}$ of the path $Me^{tJ}$ with $t\ge 0$ sufficiently small. Let
\[
\Sp(2n)_\omega^*=\Sp(2n)\setminus \Sp(2n)_\omega^0.
\]

For any two continuous curves $\xi$ and $\eta:[0,T]\to \Sp(2n)$ with $\xi(T)=\eta(0)$, the concatenation of them is defined by
\[
(\xi*\eta)(t)=
\begin{cases}
\xi(t), & 0\le t\le T,\\[2mm]
\eta(t-T), & T\le t\le 2T.
\end{cases}
\]

Given any two real matrices of the square block form
\[
M_1=
\begin{bmatrix}
A_1 & B_1\\
C_1 & D_1
\end{bmatrix}_{2k_1\times 2k_1},\qquad
M_2=
\begin{bmatrix}
A_2 & B_2\\
C_2 & D_2
\end{bmatrix}_{2k_2\times 2k_2},
\]
we define their $\diamond$-product as
\[
M_1\diamond M_2=
\begin{bmatrix}
A_1&0&B_1&0\\
0&A_2&0&B_2\\
C_1&0&D_1&0\\
0&C_2&0&D_2
\end{bmatrix}.
\]
The $k$-fold $\diamond$-product $M\diamond\cdots\diamond M$ is denoted by $M^{\diamond k}$. For any two symplectic matrix paths $\gamma_j(t)\in\mathcal{P}_T(2n_j)$ $(j=1,2)$, we define $\gamma_1\diamond\gamma_2(t)=\gamma_1(t)\diamond\gamma_2(t)$ for all $t\in[0,T]$.

Let $D(a)=\operatorname{diag}\{a,a^{-1}\}$, $M_n^+=D(2)^{\diamond n}$, $M_n^-=D(-2)\diamond D(2)^{\diamond(n-1)}$. Define a special symplectic matrix path $\xi_n\in\mathcal{P}_T(2n)$ by
\[
\xi_n(t)=
\begin{bmatrix}
2-\frac{t}{T} & 0\\
0 & (2-\frac{t}{T})^{-1}
\end{bmatrix}^{\diamond n},  0\leq t\leq T.
\]
where $\epsilon>0$ is sufficiently small.

\begin{lemma}
    (\cite{Long2002}).\label{lem:components}
    For $\omega\in U$, the set $\Sp(2n)_\omega^*$ possesses precisely two path connected components $\Sp(2n)_\omega^+$ and $\Sp(2n)_\omega^-$, and $M_n^+\in \Sp(2n)_\omega^+$ and $M_n^-\in \Sp(2n)_\omega^-$.
\end{lemma}

\begin{definition}
    (\cite{Long2002}).
\label{def:maslov_index}
For any $\omega\in U$ and $\gamma\in\mathcal{P}_T(2n)$, we define
\[
\nu_\omega(\gamma)=\dim_{\mathbb{C}}\ker_{\mathbb{C}}(\gamma(T)-\omega I_{2n})
\]
and
\[
i_\omega(\gamma)=[e^{-\epsilon J}\gamma*\xi_n:\Sp(2n)_\omega^0],
\]
where $[\cdot;\cdot]$ denotes the usual homotopy intersection number and $0<\epsilon\ll 1$.
\end{definition}

\begin{definition}
    (\cite{Long2002}).
\label{def:splitting}
For any $M\in\Sp(2n)$ and $\omega\in U$, we define the splitting numbers $S_M^\pm(\omega)$ of $M$ at $\omega$ by
\[
S_M^\pm(\omega)=\lim_{\epsilon\to 0^+}\left(i_{e^{\pm\sqrt{-1}\epsilon}\omega}(\gamma)-i_\omega(\gamma)\right)
\]
for any path $\gamma\in\mathcal{P}_T(2n)$ satisfying $\gamma(T)=M$.
\end{definition}

For $\omega=e^{\sqrt{-1}\theta}\in U$, we can define as in \cite{Long2002} the following symplectic matrices:
\begin{align}
D(\lambda)&=\begin{bmatrix}
\lambda & 0\\
0 & \lambda^{-1}
\end{bmatrix},\qquad \lambda=\pm 2, \label{eq:Dlambda}\\
N_1(\lambda,b)&=\begin{bmatrix}
\lambda & b\\
0 & \lambda
\end{bmatrix},\qquad \lambda=\pm 1,\; b=\pm 1,0, \label{eq:N1}\\
R(\theta)&=\begin{bmatrix}
\cos\theta & -\sin\theta\\
\sin\theta & \cos\theta
\end{bmatrix},\qquad \theta\in(0,\pi)\cup(\pi,2\pi), \label{eq:Rtheta}\\
N_2(\omega,b)&=\begin{bmatrix}
R(\theta) & b\\
0 & R(\theta)
\end{bmatrix},\qquad \theta\in(0,\pi)\cup(\pi,2\pi),\nonumber\\
&\quad b=\begin{bmatrix}
b_1 & b_2\\
b_3 & b_4
\end{bmatrix},\; b_i\in\mathbb{R},\; b_2\neq b_3. \label{eq:N2}
\end{align}
Usually, a symplectic matrix being one type of \eqref{eq:Dlambda}--\eqref{eq:N2} is called a basic normal form of $\omega$.

\begin{definition}\label{def:trivial-normal-form}
    (\cite{Long2002}).
A basic normal form matrix $M\in \Sp(2n)$ is trivial if for sufficiently small $\alpha>0$, $MR((t-1)\alpha)^{\diamond n}$ possesses no eigenvalues on $U$ for $t\in[0,1)$, and nontrivial otherwise.
\end{definition}

The splitting numbers have the following properties.

\begin{lemma}
    (\cite{Long2002}).
\label{lem:splitting_props}
The splitting numbers $S_M^\pm(\omega)$ are well-defined, i.e., are independent of the path $\gamma\in\mathcal{P}_T(2n)$ appearing in Definition~\ref{def:splitting}. For $\omega\in U$ and $M\in\Sp(2n)$, $S_M^\pm(\omega)$ are constant for all $N=P^{-1}MP$ with $P\in\Sp(2n)$.
\end{lemma}

\begin{lemma}
    (\cite{Long2002}).
\label{lem:splitting_formulas}
For $M\in\Sp(2n)$, $\omega\in U$, and $\theta\in(0,\pi)$, we have
\begin{align}
& S_M^\pm(\omega)=0 \quad\text{if } \omega\notin\sigma(M), \label{eq:splitting1}\\
& S_M^\pm(\omega)=S_M^\mp(\overline{\omega}), \label{eq:splitting2}\\
& 0\le S_M^\pm(\omega)\le \dim\ker(M-\omega I_{2n}), \label{eq:splitting3}\\
& S_M^+(\omega)+S_M^-(\omega)\le \dim\ker(M-\omega I_{2n})^2,\quad \omega\in\sigma(M), \label{eq:splitting4}\\
& \bigl(S_{N_1(1,b)}^+(1),\,S_{N_1(1,b)}^-(1)\bigr)=
\begin{cases}
(1,1), & \text{if } b=0,1,\\
(0,0), & \text{if } b=-1,
\end{cases} \label{eq:splitting5}\\
& \bigl(S_{N_1(-1,a)}^+(-1),\,S_{N_1(-1,a)}^-(-1)\bigr)=
\begin{cases}
(1,1), & \text{if } a=0,-1,\\
(0,0), & \text{if } a=1,
\end{cases} \label{eq:splitting6}\\
& \bigl(S_{R(\theta)}^+(e^{\sqrt{-1}\theta}),\,S_{R(\theta)}^-(e^{\sqrt{-1}\theta})\bigr)=(0,1), \label{eq:splitting7}\\
& \bigl(S_{R(2\pi-\theta)}^+(e^{\sqrt{-1}\theta}),\,S_{R(2\pi-\theta)}^-(e^{\sqrt{-1}\theta})\bigr)=(1,0), \label{eq:splitting8}\\
& \bigl(S_M^+(e^{\sqrt{-1}\theta}),\,S_M^-(e^{\sqrt{-1}\theta})\bigr)=(1,1), \label{eq:splitting9}
\end{align}
where $M=N_2(\omega,b)$ is nontrivial, and
\begin{equation}
\bigl(S_M^+(e^{\sqrt{-1}\theta}),\,S_M^-(e^{\sqrt{-1}\theta})\bigr)=(0,0), \label{eq:splitting10}
\end{equation}
where $M=N_2(\omega,b)$ is trivial.
\end{lemma}

For any $M_i\in\Sp(2n_i)$ with $i=0,1$, we have
\begin{equation}
S_{M_0\diamond M_1}^\pm(\omega)=S_{M_0}^\pm(\omega)+S_{M_1}^\pm(\omega),\qquad \forall \omega\in U. \label{eq:splitting11}
\end{equation}

The generalized Bott-type iteration formula is formulated as follows.

\begin{theorem}
    (\cite{HS2009}).
\label{thm:bott}
Let $z$ be a solution of the Hamiltonian system \eqref{31}, \eqref{32} with $\gamma$ its fundamental solution. Then, with notations as in (\cite{HS2009}),
\begin{align}
&\mu(\Gr(Q^{\top}), \Gr(\gamma(t)), t\in[0,T])\notag\\
&\quad = \sum_{i=1}^{m}\mu(\Gr\bigl(\exp(\tfrac{i}{m}2\pi\sqrt{-1})S^{\top}\bigr), \Gr(\gamma(t)), t\in[0,T/m]).
\label{eq:A-16}
\end{align}

\end{theorem}
In what follows, we recall the relation between the Maslov index and Morse index of solutions of the $N$-body problem as critical points of the corresponding action functional. Define the action functional by
\[
\Phi(x)=\int_0^T L(t,x,\dot{x})\,\mathrm{d}t,\qquad \forall x\in W^{1,2}([0,T],\mathbb{R}^n),
\]
where $L\in C^2([0,T]\times\mathbb{R}^{2n},\mathbb{R})$ and it satisfies the Legendrian convexity condition. Consider the following boundary condition:
\begin{equation}
x(\tfrac{T}{n})=\omega S^{\top} x(0),
\label{eq:maslov-graph-boundary}
\end{equation}
where $\omega\in U$ and $S$ is an $n\times n$ orthogonal matrix. We have

\begin{theorem}
    (\cite{HS2009}).
\label{thm:morse_maslov}
For a critical point $x$ of the functional $\Phi$ under boundary condition \eqref{eq:maslov-graph-boundary},
\[
m^-(x)+\nu_\omega(S)=\mu\bigl(\Gr(S_d^{\top}), \Gr(\gamma(t)), t\in[0,\tfrac{T}{n}]\bigr),
\]
where $m^-(x)$ is the Morse index of $\Phi$ at $x$, i.e., $m^-(x)=m^-(\Phi''(x))$, $\nu_\omega(S)=\dim\ker(\overline{\omega}S-I_n)$, and $S_d=\begin{bmatrix} S&0\\0&S \end{bmatrix}$.

\end{theorem}

\section{Numerical computation of the Morse index}\label{app:numerical-index}
We compute the Morse index of a prescribed collision-free double--double orbit by discretizing the second variation of the Lagrangian action on a space of variations satisfying a symmetry-induced graph boundary condition. The boundary constraint is imposed by eliminating dependent endpoint variables. The resulting symmetric matrix is then analyzed through a generalized eigenvalue problem.

\paragraph{Lagrangian formulation and graph boundary conditions.}
In the parallelogram configuration, the four positions are
$(q_1,q_2,-q_2,-q_1)$, with corresponding masses $(1,m,m,1)$.
We use the independent configuration variable
$q=(q_1^{\top},q_2^{\top})^{\top}\in\mathbb{R}^4$ and define
\begin{equation}
P=2\operatorname{diag}(1,1,m,m).
\label{eq:B-1}
\end{equation}
The Lagrangian and Hamiltonian are
\begin{equation}
L(q,\dot q)=\frac12\dot q^{\top}P\dot q+U(q),
\qquad
\mathcal{H}(q,p)=\frac12p^{\top}P^{-1}p-U(q),
\label{eq:dd-lagrangian}
\end{equation}
where $p=P\dot q$ is the canonical momentum and
\begin{equation}
U(q)=
\frac{1}{2\lVert q_1\rVert}
+\frac{m^2}{2\lVert q_2\rVert}
+\frac{2m}{\lVert q_1+q_2\rVert}
+\frac{2m}{\lVert q_1-q_2\rVert}.
\label{eq:dd-potential}
\end{equation}
Consequently, the equations used to reconstruct the orbit are
\begin{equation}
\dot q=P^{-1}p,
\qquad
\dot p=\nabla U(q).
\label{eq:dd-hamiltonian-system}
\end{equation}

Let $\theta=\pi/n$, where $n\geq2$ is an integer, and set
\begin{equation}
R(\alpha)=
\begin{pmatrix}
\cos\alpha&-\sin\alpha\\
\sin\alpha& \cos\alpha
\end{pmatrix},
\qquad
A_T=\operatorname{diag}\bigl(R(T\theta),R(T\theta)\bigr).
\label{eq:dd-boundary-matrix}
\end{equation}
The double--double symmetry described in \cite{Peng2015} gives
$q(t+2)=A_2q(t)$ and $p(t+2)=A_2p(t)$.
Thus, for $T=2k$ with $k$ a positive integer, the orbit satisfies
\begin{equation}
q(T)=A_Tq(0),
\qquad
p(T)=A_Tp(0).
\label{eq:dd-orbit-boundary}
\end{equation}

For each such $T$, consider the action
\begin{equation}
\mathcal{A}_T(q)=\int_0^T L(q,\dot q)\,\mathrm{d}t
\label{eq:B-7}
\end{equation}
and the variation space
\begin{equation}
E_T=
\left\{
\xi\in H^1([0,T],\mathbb{R}^4):
\xi(T)=A_T\xi(0)
\right\}.
\label{eq:dd-variation-space}
\end{equation}
This is the graph boundary condition
$\Gr(A_T)$ in configuration space.
In the notation of \cite{HS2009}, it corresponds to
$\Gr(\bar S^{\top})$ with $\bar S=A_T^{\top}$.

Only the position constraint is imposed on the $H^1$ variation space.
Indeed, the boundary term in the first variation is
\begin{equation}
p(T)^{\top}\xi(T)-p(0)^{\top}\xi(0)
=
\bigl(A_T^{\top}p(T)-p(0)\bigr)^{\top}\xi(0).
\label{eq:B-9}
\end{equation}
Since $A_T$ is orthogonal, its vanishing yields the natural momentum
condition in \eqref{eq:dd-orbit-boundary}.

The principal intervals are $[0,2]$, $[0,4]$, and $[0,2n]$.
For even $n$, we also consider $[0,n]$, for which $A_n=-I_4$
and the boundary condition is antiperiodic.
At the full period, $A_{2n}=I_4$, so the boundary condition is periodic.
Each interval is treated separately: no additional endpoint-type
constraint is imposed at its interior nodes.
All indices are computed within the parallelogram configuration space,
without fixing angular momentum or quotienting out overall rotations.

\paragraph{Reconstruction and verification of the reference orbit.}

The initial state is parameterized by
\begin{equation}
Y_0(a,b,c,d)=
(a,0,b,0,0,c,0,d)^{\top},
\qquad a>b>0,
\label{eq:B-10}
\end{equation}
in the coordinate ordering $Y=(q,p)$.
Starting from a supplied initial guess, the parameters are refined by
a nonlinear least-squares shooting procedure.
Writing
\[
e_{\theta}=(\cos\theta,\sin\theta)^{\top},
\qquad
e_{\theta}^{\perp}=(-\sin\theta,\cos\theta)^{\top},
\]
the shooting conditions at $t=1$ are
\begin{equation}
(e_{\theta}^{\perp})^{\top}q_i(1)=0,
\qquad
e_{\theta}^{\top}p_i(1)=0,
\qquad i=1,2.
\label{eq:dd-shooting}
\end{equation}
The position and momentum residuals are normalized separately by
$\max\{1,\lVert q(1)\rVert_{\infty}\}$ and
$\max\{1,\lVert p(1)\rVert_{\infty}\}$, respectively.
The refined initial condition is accepted only if the infinity norm
of the normalized shooting residual does not exceed $10^{-9}$.

The orbit is integrated using MATLAB's \texttt{ode113}, with relative
and absolute tolerances $2\times10^{-13}$ and $2\times10^{-14}$.
During the Hessian computation, the maximum integration step is
$\min\{10^{-3},h\}$, where $h$ is the finite-element mesh size.
For every interval, we additionally evaluate
\begin{equation}
r_q(T)=
\frac{\lVert q(T)-A_Tq(0)\rVert_{\infty}}
{\max\{1,\lVert q(0)\rVert_{\infty},\lVert q(T)\rVert_{\infty}\}},
\qquad
r_p(T)=
\frac{\lVert p(T)-A_Tp(0)\rVert_{\infty}}
{\max\{1,\lVert p(0)\rVert_{\infty},\lVert p(T)\rVert_{\infty}\}}.
\label{eq:dd-boundary-residuals}
\end{equation}
The index calculation is performed only when
$\max\{r_q(T),r_p(T)\}\leq10^{-7}$.

\paragraph{Second variation and the continuous Morse index.}

Along the reference orbit, let
\begin{equation}
B(t)=D^2U(q(t)).
\label{eq:B-13}
\end{equation}
The second variation restricted to $E_T$ is
\begin{equation}
\mathcal{Q}_T(\xi,\eta)
=
\int_0^T
\left[
\dot\xi^{\top}P\dot\eta
+\xi^{\top}B(t)\eta
\right]\,\mathrm{d}t,
\qquad \xi,\eta\in E_T.
\label{eq:dd-index-form}
\end{equation}
The positive sign of the potential term follows from the convention
$L=K+U$ in \eqref{eq:dd-lagrangian}.

The Morse index is defined by
\begin{equation}
m_T^-=
\sup\left\{
\dim W:
W\subset E_T,\
\mathcal{Q}_T(\xi,\xi)<0
\text{ for every }0\neq\xi\in W
\right\}.
\label{eq:dd-morse-definition}
\end{equation}
Equivalently, it is the number of negative eigenvalues, counted with
multiplicity, of the self-adjoint operator
\begin{equation}
\mathcal{L}_T\xi=-P\ddot\xi+B(t)\xi
\label{eq:B-16}
\end{equation}
with domain
\begin{equation}
D(\mathcal{L}_T)=
\left\{
\xi\in H^2([0,T],\mathbb{R}^4):
\xi(T)=A_T\xi(0),\
P\dot\xi(T)=A_TP\dot\xi(0)
\right\}.
\label{eq:dd-operator-domain}
\end{equation}
The derivative condition is a natural operator boundary condition;
it is not imposed as an additional constraint on the finite-element
trial functions.

The potential Hessian is evaluated analytically.
For $x\neq0$, define
\begin{equation}
D(x)=D^2\bigl(\lVert x\rVert^{-1}\bigr)
=
-\frac{I_2}{\lVert x\rVert^3}
+\frac{3xx^{\top}}{\lVert x\rVert^5}.
\label{eq:B-18}
\end{equation}
With
\[
B_1=\frac12D(q_1),\qquad
B_2=\frac{m^2}{2}D(q_2),\qquad
B_{\pm}=2mD(q_1\pm q_2),
\]
we obtain
\begin{equation}
B(t)=
\begin{pmatrix}
B_1+B_++B_- & B_+-B_-\\
B_+-B_- & B_2+B_++B_-
\end{pmatrix}.
\label{eq:dd-potential-hessian}
\end{equation}

\paragraph{Finite-element discretization.}

Partition $[0,T]$ uniformly into $N$ elements, with
$t_j=jh$ and $h=T/N$.
If $r$ denotes the number of elements per time interval of length
$1/2$, then
\begin{equation}
N=2Tr,\qquad h=\frac{1}{2r}.
\label{eq:B-20}
\end{equation}
The current implementation uses
$r\in\{30,50,100,150,200\}$ to examine mesh dependence.

We approximate variations by continuous piecewise linear functions.
On the element $[t_j,t_{j+1}]$, write
\begin{equation}
\xi_h(t_j+hs)=
\Phi(s)
\begin{pmatrix}
\xi_j\\ \xi_{j+1}
\end{pmatrix},
\qquad
\Phi(s)=\begin{pmatrix}(1-s)I_4&sI_4\end{pmatrix},
\qquad 0\leq s\leq1.
\label{eq:B-21}
\end{equation}
The kinetic contribution is integrated exactly, whereas the potential
contribution is approximated by three-point Gauss--Legendre quadrature:
\begin{equation}
(s_1,s_2,s_3)
=
\left(
\frac12-\frac{\sqrt{15}}{10},
\frac12,
\frac12+\frac{\sqrt{15}}{10}
\right),
\qquad
(w_1,w_2,w_3)=
\left(\frac5{18},\frac49,\frac5{18}\right).
\label{eq:B-22}
\end{equation}
Let $\widetilde q(t)$ denote the continuous numerical approximation
provided by the ODE solver, and set
$B_{j,g}=D^2U(\widetilde q(t_j+hs_g))$.
The element Hessian is
\begin{equation}
H_j=
\frac1h
\begin{pmatrix}
P&-P\\
-P&P
\end{pmatrix}
+
h\sum_{g=1}^{3}
w_g\,\Phi(s_g)^{\top}B_{j,g}\Phi(s_g).
\label{eq:dd-element-hessian}
\end{equation}
The coefficients $B_{j,g}$ are evaluated using \texttt{deval}
on the ODE solution. Thus, piecewise linear interpolation is used
for the variations, while the reference trajectory is represented
by the ODE solver's continuous output.

Assembly of \eqref{eq:dd-element-hessian} over all elements gives
a sparse symmetric matrix
$H_{\mathrm{full}}\in\mathbb{R}^{4(N+1)\times4(N+1)}$.
The corresponding $L^2$ Gram matrix is assembled from
\begin{equation}
G_j=\frac{h}{6}
\begin{pmatrix}
2I_4&I_4\\
I_4&2I_4
\end{pmatrix}.
\label{eq:dd-element-gram}
\end{equation}
This Gram matrix represents the $L^2$ inner product and is distinct
from the physical mass matrix $P$.

\paragraph{Elimination of the graph boundary constraint.}

The endpoint relation $\xi_N=A_T\xi_0$ is imposed algebraically.
For the independent coefficient vector
\[
v=(v_0^{\top},\ldots,v_{N-1}^{\top})^{\top}\in\mathbb{R}^{4N},
\]
define the injective map
$C_T:\mathbb{R}^{4N}\to\mathbb{R}^{4(N+1)}$ by
\begin{equation}
C_Tv=
\begin{pmatrix}
v_0\\
v_1\\
\vdots\\
v_{N-1}\\
A_Tv_0
\end{pmatrix}.
\label{eq:B-25}
\end{equation}
The constrained Hessian and Gram matrices are
\begin{equation}
H_{T,h}=C_T^{\top}H_{\mathrm{full}}C_T,
\qquad
G_{T,h}=C_T^{\top}G_{\mathrm{full}}C_T.
\label{eq:dd-constrained-matrices}
\end{equation}
This restriction preserves the second variation on the prescribed
finite-element space and leaves $4N$ independent degrees of freedom.
In particular, the initial endpoint remains variable and is coupled
to the terminal endpoint through $A_T$.

\paragraph{Generalized eigenvalues and index counting.}

Since $G_{T,h}$ is positive definite, we solve
\begin{equation}
H_{T,h}v=\lambda G_{T,h}v.
\label{eq:dd-generalized-eigenproblem}
\end{equation}
In exact arithmetic, the number of negative generalized eigenvalues
equals the negative inertia of $H_{T,h}$, because
\begin{equation}
n_-(H_{T,h})
=
n_-\!\left(
G_{T,h}^{-1/2}H_{T,h}G_{T,h}^{-1/2}
\right).
\label{eq:B-28}
\end{equation}
The use of the $L^2$ Gram matrix also makes the computed low-lying
eigenvalues suitable for comparison across different meshes.

For matrices of dimension at most $1200$, the full generalized
spectrum is computed. Larger problems are treated using a sparse
shift-and-invert eigensolver. To ensure that the computation targets
the bottom of the spectrum, define
\begin{equation}
\beta_h=\min_{j,g}\lambda_{\min}(B_{j,g}).
\label{eq:B-29}
\end{equation}
The quadrature weights are positive, and the quadrature is exact
for products of the linear shape functions. Since the kinetic
contribution is nonnegative, the assembled matrices satisfy
\begin{equation}
H_{T,h}\succeq\beta_hG_{T,h}
\label{eq:dd-spectral-lower-bound}
\end{equation}
in exact arithmetic.

The initial spectral shift is chosen as
\begin{equation}
\sigma=
\min\{0,\beta_h\}
-\max\{1,0.1|\beta_h|\},
\label{eq:B-31}
\end{equation}
and is moved further to the left if necessary.
Positive definiteness of $H_{T,h}-\sigma G_{T,h}$ is checked numerically
by Cholesky factorization.
For a shift below the spectrum, the eigenvalues of
\[
(H_{T,h}-\sigma G_{T,h})^{-1}G_{T,h}
\]
are $(\lambda-\sigma)^{-1}$, so its dominant eigenvalues correspond
to the smallest algebraic generalized eigenvalues in
\eqref{eq:dd-generalized-eigenproblem}.

The sparse computation begins with $32$ eigenvalues and increases
the requested number until the largest returned eigenvalue exceeds
the classification threshold. This provides a numerical check that
the negative spectrum and the threshold-defined near-zero cluster
have been covered. The implementation permits up to $512$ requested
eigenvalues and rejects a computation if convergence or spectral
coverage is not achieved.
The eigensolver tolerance is $10^{-11}$.
For each returned sparse eigenpair, the additional residual test is
\begin{equation}
\frac{
\lVert H_{T,h}v-\lambda G_{T,h}v\rVert_2
}{
\bigl(\lVert H_{T,h}\rVert_1
+|\lambda|\lVert G_{T,h}\rVert_1\bigr)\lVert v\rVert_2
}
\leq10^{-8}.
\label{eq:dd-eigenpair-residual}
\end{equation}

Given a classification threshold $\varepsilon>0$, the reported
numerical index and near-zero count are
\begin{equation}
m_{T,h,\varepsilon}^-=
\#\{\lambda_j<-\varepsilon\},
\qquad
\nu_{T,h,\varepsilon}=
\#\{|\lambda_j|\leq\varepsilon\},
\label{eq:dd-numerical-index}
\end{equation}
with multiplicities included.
The active threshold in the present implementation is
$\varepsilon=10^{-10}$.

For example, for the prescribed parameters \(n=4\),
the initial condition of the retrograde double-double orbit is first corrected
by a shooting method. The resulting shooting residual is
\[
\lVert R_{\mathrm{shoot}}\rVert_\infty
=
2.7754\times10^{-15},
\]
and the closure error over one full period is
\[
\lVert Y(2n)-Y(0)\rVert_\infty
=
5.1356\times10^{-12}.
\]

We then compute the Morse index on the interval
\([0,2]\),\([0,4]\), \([0,8]\) within the parallelogram-reduced configuration space.
On each time interval of length \(0.5\), we use
\[
r=30,\ 50,\ 100,\ 150,\ 200
\]
finite elements, respectively. The resulting discrete Morse indexs
are listed in Table~\ref{tab:graph_morse_n4}.

\Needspace{28\baselineskip}
\begingroup
\footnotesize
\setlength{\tabcolsep}{4pt}
\renewcommand{\arraystretch}{1.12}
\begin{longtable}{@{}c@{\hspace{8pt}}rrrrrS[table-format=1.6e-2]S[table-format=1.6e-2]@{}}
\caption{Mesh-refinement results for the graph-boundary Morse index
of the retrograde double-double orbit with $m=1$ and $n=4$.}
\label{tab:graph_morse_n4}
\\
\toprule
Interval &
\(T\) &
\(N/\mathrm{half}\) &
GraphIndex &
NearZero &
Dimension &
{\(Q\)-Residual} &
{\(P\)-Residual}
\\
\midrule
\endfirsthead

\toprule
Interval &
\(T\) &
\(N/\mathrm{half}\) &
GraphIndex &
NearZero &
Dimension &
{\(Q\)-Residual} &
{\(P\)-Residual}
\\
\midrule
\endhead

\midrule
\multicolumn{8}{r}{\textit{Continued on next page}}\\
\endfoot

\bottomrule
\endlastfoot

\([0,2]\)  & 2 & 30  & 0 & 0 & 480   & 1.385594e-13 & 1.763152e-13\\
\([0,2]\)  & 2 & 50  & 0 & 0 & 800   & 1.385594e-13 & 1.763152e-13\\
\([0,2]\)  & 2 & 100 & 0 & 0 & 1600  & 1.385594e-13 & 1.763152e-13\\
\([0,2]\)  & 2 & 150 & 0 & 0 & 2400  & 1.385594e-13 & 1.763152e-13\\
\([0,2]\)  & 2 & 200 & 0 & 0 & 3200  & 1.385594e-13 & 1.763152e-13\\

\addlinespace

\([0,4]\)  & 4 & 30  & 2 & 0 & 960   & 4.402819e-13 & 5.468232e-13\\
\([0,4]\)  & 4 & 50  & 2 & 0 & 1600  & 4.402819e-13 & 5.468232e-13\\
\([0,4]\)  & 4 & 100 & 2 & 0 & 3200  & 4.402819e-13 & 5.468232e-13\\
\([0,4]\)  & 4 & 150 & 2 & 0 & 4800  & 4.402819e-13 & 5.468232e-13\\
\([0,4]\)  & 4 & 200 & 2 & 0 & 6400  & 4.402819e-13 & 5.468232e-13\\

\addlinespace

\([0,n]\)  & 4 & 30  & 2 & 0 & 960   & 4.402819e-13 & 5.468232e-13\\
\([0,n]\)  & 4 & 50  & 2 & 0 & 1600  & 4.402819e-13 & 5.468232e-13\\
\([0,n]\)  & 4 & 100 & 2 & 0 & 3200  & 4.402819e-13 & 5.468232e-13\\
\([0,n]\)  & 4 & 150 & 2 & 0 & 4800  & 4.402819e-13 & 5.468232e-13\\
\([0,n]\)  & 4 & 200 & 2 & 0 & 6400  & 4.402819e-13 & 5.468232e-13\\

\addlinespace

\([0,2n]\) & 8 & 30  & 4 & 0 & 1920  & 1.269667e-12 & 1.571970e-12\\
\([0,2n]\) & 8 & 50  & 4 & 0 & 3200  & 1.269667e-12 & 1.571970e-12\\
\([0,2n]\) & 8 & 100 & 4 & 0 & 6400  & 1.269667e-12 & 1.571970e-12\\
\([0,2n]\) & 8 & 150 & 4 & 0 & 9600  & 1.269667e-12 & 1.571970e-12\\
\([0,2n]\) & 8 & 200 & 4 & 0 & 12800 & 1.269667e-12 & 1.571970e-12\\

\end{longtable}
\endgroup

We further examine the convergence of the low-lying
generalized eigenvalues of the discrete Hessian for the
retrograde double-double orbit with $m=1$ and $n=4$.

We consider the intervals
\(
[0,2],[0,4],[0,8],
\)
equipped with the graph boundary conditions
\[
\eta(T)=A_T\eta(0),
\qquad
A_T=
\operatorname{diag}
\bigl(R(T\theta),R(T\theta)\bigr),
\qquad
\theta=\frac{\pi}{4}.
\]

In particular,
\[
A_2=\operatorname{diag}
\left(R\left(\frac{\pi}{2}\right),
R\left(\frac{\pi}{2}\right)\right),
\qquad
A_4=-I_4,
\qquad
A_8=I_4.
\]

\paragraph{Symmetry-induced null modes.}

Before presenting the numerical results, we recall
that the continuous second variation possesses
null directions arising from the symmetries
of the Newtonian system.

Let
\[
\mathcal M=\operatorname{diag}(1,1,1,1)
\]
be the mass matrix and let $q(t)$ denote the
retrograde double-double orbit.

The self-adjoint operator associated with the
second variation is
\[
\mathcal L\eta
=
-2\mathcal M\ddot{\eta}
+
D^2U(q(t))\eta.
\]

\begin{lemma}\label{lem:nullity-full-period}
Let $q(t)$ be a retrograde double-double periodic orbit
with period $T=2n$, and let $\gamma(t)\in\Sp(8)$ be the
fundamental solution of the corresponding linearized
Hamiltonian system, satisfying $\gamma(0)=I_8$.

From \eqref{2} and \textit{nondegenerate} assumption, we know the monodromy matrix $\gamma(T)$ and $A_T^{\top}\gamma(\frac{T}{n})$ , $(A_T^{\top})^2\gamma(\frac{2T}{n})$ are
symplectically similar to
\[
N_1(1,1)\diamond N_1(1,b)\diamond M_2,
\]
where $b=1$, $M_2\in\Sp(4)$, and
\[
1\notin\sigma(M_2).
\]
Then the nullity of the second variation of the
Lagrangian action functional under the graph
boundary condition is exactly two:
\[
\nu_{t}=\dim\ker Q_{t}=2, \qquad t=T,\frac{2T}{n},\frac{T}{n}.
\]
\end{lemma}

\begin{proof}
We only consider $\gamma(T)$ and the others are similar. Since $q(t)$ is periodic with period $2n$, the
graph boundary condition on the full period
reduces to the periodic boundary condition.

By the correspondence between the nullity of
the Lagrangian second variation and that of
the linearized Hamiltonian boundary-value problem(\cite{HS2009}, Lemma 3.1), 
we have
\begin{equation}\label{eq:nullity-monodromy}
\nu_{2n}
=
\dim\ker Q_{2n}
=
\dim\ker\bigl(\gamma(2n)-I_8\bigr).
\end{equation}
By \eqref{2}, there exists a symplectic matrix
$P\in\Sp(8)$ such that
\[
P^{-1}\gamma(2n)P
=
N_1(1,1)\diamond N_1(1,b)\diamond M_2.
\]
Recall that
\[
N_1(1,b)
=
\begin{pmatrix}
1&b\\
0&1
\end{pmatrix}.
\]
Since $b=1$, both Jordan blocks
$N_1(1,1)$ have a one-dimensional
eigenspace corresponding to the eigenvalue $1$.
More precisely,
\[
\dim\ker\bigl(N_1(1,1)-I_2\bigr)=1.
\]
Moreover, since $1\notin\sigma(M_2)$,
\[
\dim\ker(M_2-I_4)=0.
\]

Consequently, the geometric multiplicity of
the eigenvalue $1$ of $\gamma(2n)$ is
\[
\begin{aligned}
\dim\ker\bigl(\gamma(2n)-I_8\bigr)
&=
\dim\ker\bigl(N_1(1,1)-I_2\bigr)\\
&\quad+
\dim\ker\bigl(N_1(1,1)-I_2\bigr)\\
&\quad+
\dim\ker(M_2-I_4)\\
&=1+1+0\\
&=2.
\end{aligned}
\]
Combining this with
\eqref{eq:nullity-monodromy}, we conclude that
\[
\nu_{2n}=\dim\ker Q_{2n}=2.
\]
\end{proof}
The preceding lemma establishes that the continuous
Hessian possesses exactly two zero eigenvalues.
 In the following, we investigate how these zero eigenvalues
are approximated by the finite-element discretization.
 Let
\[
H_hv=\lambda G_hv
\]
be the generalized eigenvalue problem associated
with the discrete second variation.
 As shown in Table~\ref{tab:n4-zero-modes},
two generalized eigenvalue branches converge
toward zero as the mesh is refined. 
Furthermore, the scaled quantities
\[
N_{\mathrm{half}}^2\lambda_{1,h},
\qquad
N_{\mathrm{half}}^2\lambda_{2,h}
\]
remain approximately constant, indicating
quadratic convergence with respect to the
mesh size $h$. 
This behavior is consistent with the
two-dimensional kernel of the continuous
Hessian established above.

\paragraph{Convergence of the two near-zero eigenvalues.}

Let
\[
H_hv=\lambda G_hv
\]
be the generalized eigenvalue problem
associated with the finite-element
discretization, where $H_h$ is the discrete
Hessian and $G_h$ is the $L^2$ Gram matrix.

The mesh size is
\[
h=\frac{1}{2N_{\mathrm{half}}},
\]
where $N_{\mathrm{half}}$ denotes the number
of finite elements per time interval of
length $0.5$.

We first consider the interval $[0,2]$.
The two smallest generalized eigenvalues,
denoted by $\lambda_{1,h}$ and
$\lambda_{2,h}$, are reported in
Table~\ref{tab:n4-zero-modes}.
To examine their convergence rates, we
also report the scaled quantities
\[
N_{\mathrm{half}}^2\lambda_{1,h},
\qquad
N_{\mathrm{half}}^2\lambda_{2,h}.
\]

\begin{table}[htbp]
\centering
\small
\caption{
Mesh convergence of the two smallest
generalized eigenvalues on $[0,2]$
for the retrograde double-double orbit
with $m=1$ and $n=4$.
}
\label{tab:n4-zero-modes}

\begin{tabular}{ccccc}
\toprule
$N_{\mathrm{half}}$
&
$\lambda_{1,h}$
&
$N_{\mathrm{half}}^2\lambda_{1,h}$
&
$\lambda_{2,h}$
&
$N_{\mathrm{half}}^2\lambda_{2,h}$
\\
\midrule

30
&
\num{2.6778089e-5}
&
\num{0.02410028}
&
\num{1.4660459e-3}
&
\num{1.31944127}
\\

50
&
\num{9.6411549e-6}
&
\num{0.02410289}
&
\num{5.2776648e-4}
&
\num{1.31941619}
\\

100
&
\num{2.4104134e-6}
&
\num{0.02410413}
&
\num{1.3194058e-4}
&
\num{1.31940581}
\\

150
&
\num{1.0713414e-6}
&
\num{0.02410518}
&
\num{5.8640208e-5}
&
\num{1.31940469}
\\

200
&
\num{6.0261218e-7}
&
\num{0.02410449}
&
\num{3.2985081e-5}
&
\num{1.31940324}
\\

\bottomrule
\end{tabular}
\end{table}

As shown in Table~\ref{tab:n4-zero-modes},
both eigenvalues decrease toward zero
under mesh refinement.
More importantly, the scaled quantities
$N_{\mathrm{half}}^2\lambda_{1,h}$
and
$N_{\mathrm{half}}^2\lambda_{2,h}$
remain nearly constant.
Numerically, we observe
\[
N_{\mathrm{half}}^2\lambda_{1,h}
\approx 0.024104,
\]
and
\[
N_{\mathrm{half}}^2\lambda_{2,h}
\approx 1.319403.
\]

Since
\[
h=\frac{1}{2N_{\mathrm{half}}},
\]
these observations are consistent with
the asymptotic relations
\[
\lambda_{1,h}\approx 0.09642h^2,
\qquad
\lambda_{2,h}\approx 5.27761h^2.
\]
In particular, both branches exhibit
approximately quadratic convergence
toward zero:
\[
\lambda_{1,h}=O(h^2),
\qquad
\lambda_{2,h}=O(h^2).
\]

An empirical convergence order can also
be estimated by
\[
p_j=
\frac{
\log\bigl(
\lambda_{j,h_{30}}/
\lambda_{j,h_{200}}
\bigr)
}{
\log(h_{30}/h_{200})
},
\qquad j=1,2.
\]
Using the numerical values in
Table~\ref{tab:n4-zero-modes}, we obtain
\[
p_1\approx1.9999,
\qquad
p_2\approx2.0000.
\]

These results provide strong numerical
evidence that the two small generalized
eigenvalues approximate the two
symmetry-induced null directions
of the continuous Hessian.

The same two near-zero eigenvalue branches
are also observed on $[0,4]$ and $[0,8]$,
with almost identical numerical values
at the same mesh resolution.

This agreement is consistent with the
rotational time-shift symmetry of the orbit.

\paragraph{Convergence of the nonzero spectrum.}

To distinguish the near-zero eigenvalues
from the genuinely nonzero spectrum,
we further examine the neighboring
negative and positive eigenvalues.

Let
\[
\lambda_{1,h}^{(T)}
\leq
\lambda_{2,h}^{(T)}
\leq\cdots
\]
denote the generalized eigenvalues
on the interval $[0,T]$,
arranged in nondecreasing order.

The numerical results show the following
spectral structures:

\[
\begin{array}{c|c|c}
\text{Interval}
&
\text{Number of negative eigenvalues}
&
\text{Near-zero branches}
\\
\hline
[0,2] & 0 & \lambda_{1,h}^{(2)},
                \lambda_{2,h}^{(2)}
\\

[0,4] & 2 & \lambda_{3,h}^{(4)},
                \lambda_{4,h}^{(4)}
\\

[0,8] & 4 & \lambda_{5,h}^{(8)},
                \lambda_{6,h}^{(8)}
\end{array}
\]
In the following, Table~\ref{tab:n4-nonzero-spectrum}
presents selected nonzero eigenvalues
adjacent to the near-zero spectral branches.

\begin{table}[htbp]
\centering
\footnotesize

\caption{
Mesh convergence of selected nonzero
generalized eigenvalues on $[0,2]$,
$[0,4]$, and $[0,8]$.
The superscript indicates the interval endpoint.
}
\label{tab:n4-nonzero-spectrum}

\begin{tabular}{cccccc}
\toprule
$N_{\mathrm{half}}$
&
$\lambda_{3,h}^{(2)}$
&
$\lambda_{2,h}^{(4)}$
&
$\lambda_{5,h}^{(4)}$
&
$\lambda_{4,h}^{(8)}$
&
$\lambda_{7,h}^{(8)}$
\\
\midrule

30
&
\num{1.7118263}
&
\num{-0.6048519}
&
\num{1.5557894}
&
\num{-0.3044381}
&
\num{0.02533484}
\\

50
&
\num{1.7100222}
&
\num{-0.6061035}
&
\num{1.5557385}
&
\num{-0.3054526}
&
\num{0.02530230}
\\

100
&
\num{1.7092609}
&
\num{-0.6066315}
&
\num{1.5557171}
&
\num{-0.3058806}
&
\num{0.02528859}
\\

150
&
\num{1.7091199}
&
\num{-0.6067293}
&
\num{1.5557131}
&
\num{-0.3059599}
&
\num{0.02528605}
\\

200
&
\num{1.7090706}
&
\num{-0.6067635}
&
\num{1.5557117}
&
\num{-0.3059876}
&
\num{0.02528516}
\\

\bottomrule
\end{tabular}
\end{table}

On $[0,2]$, the first two eigenvalues
approach zero, whereas the third
eigenvalue converges toward a strictly
positive value, approximately $1.709$.

On $[0,4]$, the first two eigenvalues
remain negative and converge toward
approximately $-0.607$.
They are followed by two near-zero
eigenvalue branches, while the fifth
eigenvalue approaches approximately
$1.556$.

Similarly, on the full period $[0,8]$,
the first four eigenvalues remain
strictly negative.
The fifth and sixth eigenvalues
approach zero, whereas the seventh
eigenvalue converges toward approximately
$0.0253$.

These observations indicate a clear
numerical separation between the
symmetry-induced near-zero eigenvalues
and the remaining nonzero spectrum.

In particular, the negative eigenvalue
counts remain unchanged throughout
the mesh-refinement study:
\[
m_h^-([0,2])=0,
\qquad
m_h^-([0,4])=2,
\qquad
m_h^-([0,8])=4.
\]

\paragraph{Detailed low-lying spectrum.}

For completeness, Table~\ref{tab:n4-low-spectrum}
lists the twelve smallest generalized
eigenvalues obtained on each of the
three intervals using the finest mesh,
$N_{\mathrm{half}}=200$.

\begin{table}[htbp]
\centering
\small

\caption{
The twelve smallest generalized eigenvalues
on $[0,2]$, $[0,4]$, and $[0,8]$
for the retrograde double-double orbit
with $m=1$, $n=4$, and
$N_{\mathrm{half}}=200$.
}
\label{tab:n4-low-spectrum}

\begin{tabular}{cccc}
\toprule
$j$
&
$\lambda_{j,h}^{(2)}$
&
$\lambda_{j,h}^{(4)}$
&
$\lambda_{j,h}^{(8)}$
\\
\midrule

1
&
\num{6.0261218e-7}
&
\num{-6.0676351e-1}
&
\num{-6.0676351e-1}
\\

2
&
\num{3.2985081e-5}
&
\num{-6.0676351e-1}
&
\num{-6.0676351e-1}
\\

3
&
\num{1.7090706}
&
\num{6.0260843e-7}
&
\num{-3.0598763e-1}
\\

4
&
\num{3.6607417}
&
\num{3.2985079e-5}
&
\num{-3.0598763e-1}
\\

5
&
\num{3.7855613}
&
\num{1.5557117}
&
\num{6.0261136e-7}
\\

6
&
\num{1.1566441e1}
&
\num{1.5557117}
&
\num{3.2985083e-5}
\\

7
&
\num{1.1640056e1}
&
\num{1.7090706}
&
\num{2.5285160e-2}
\\

8
&
\num{3.1151867e1}
&
\num{3.6607417}
&
\num{2.5285160e-2}
\\

9
&
\num{3.2322563e1}
&
\num{3.7855613}
&
\num{8.6106367e-2}
\\

10
&
\num{3.3448918e1}
&
\num{1.1566441e1}
&
\num{8.6106367e-2}
\\

11
&
\num{3.3987556e1}
&
\num{1.1640056e1}
&
\num{1.5557117}
\\

12
&
\num{3.6005124e1}
&
\num{1.1784553e1}
&
\num{1.5557117}
\\

\bottomrule
\end{tabular}
\end{table}

The eigenvalue distributions in
Table~\ref{tab:n4-low-spectrum}
are consistent with the existence
of two symmetry-induced null modes
in each graph-boundary problem.
The two near-zero eigenvalues appear
at different positions in the ordered
spectra because the numbers of
negative eigenvalues differ among
the three intervals.
More precisely, they correspond to
the first and second eigenvalues
on $[0,2]$, the third and fourth
on $[0,4]$, and the fifth and sixth
on $[0,8]$.
We emphasize that the finite-element
eigenvalues associated with the
continuous null modes need not be
exactly zero on a finite mesh.
Their observed quadratic decay is
consistent with the expected
convergence behavior of the
piecewise-linear finite-element
approximation.

\paragraph{Conclusion.}

The numerical results exhibit three
consistent features:

\begin{enumerate}
\item
Two generalized eigenvalue branches
converge toward zero at an approximately
quadratic rate under mesh refinement.

\item
The remaining neighboring negative
and positive eigenvalues converge
toward values separated from zero.

\item
The negative eigenvalue counts remain
constant across all five mesh resolutions.
\end{enumerate}

Together with the analytically established
symmetry-induced null directions,
these results provide strong numerical
evidence for the reliability of the
computed graph-boundary Morse indices.

For the representative case $m=1$,
$n=4$, the computations consistently give
\[
m^-([0,2])=0,\qquad
m^-([0,4])=2,\qquad
m^-([0,8])=4
\]
as the numerically inferred Morse indices.

\hfill\newline
\noindent{\bf Acknowledgement.}
The author is extremely grateful to Professor Yuwei Ou who introduced him to the stability problem in celestial mechanics and thanks him for his precious help and helpful discussions. The author sincerely thanks Professor Xijun Hu for his suggestions to this paper. The author also thanks Professor Duokui Yan for some numerical data and valuable comments he provided during the course of this research.
The author is partially supported by NSFC (No.12371192).

\end{document}